\documentclass[11pt,a4paper,reqno]{amsart}
 \usepackage{amsfonts,amsmath,amscd,amssymb,amsbsy,amsthm,amstext,amsopn,amsxtra}
 \usepackage{color,fullpage,mathrsfs,verbatim}
 \usepackage{stmaryrd}
 \usepackage{extarrows}
 \usepackage[all]{xy}
 \usepackage{bm}   
 \usepackage{todonotes}
 \usepackage{hyperref}
 \usepackage{tikz}
 \usetikzlibrary{shapes.geometric, calc}
 \usepackage{enumitem}
\usetikzlibrary{arrows,decorations.pathmorphing,decorations.pathreplacing,positioning,shapes.geometric,shapes.misc,decorations.markings,decorations.fractals,calc,patterns}
 \hypersetup{colorlinks,linkcolor=[rgb]{0,0,1},citecolor=[rgb]{1,0,0}}
 
 \tikzset{>=stealth',
        cvertex/.style={circle,draw=black,inner sep=1pt,outer sep=3pt},
        vertex/.style={circle,fill=black,inner sep=1pt,outer sep=3pt},
        star/.style={circle,fill=yellow,inner sep=0.75pt,outer sep=0.75pt},
        tvertex/.style={inner sep=1pt,font=\scriptsize},
        gap/.style={inner sep=0.5pt,fill=white}}

 \numberwithin{equation}{section}
\newtheorem{theorem}{Theorem}[section]
\newtheorem{lemma}[theorem]{Lemma}
\newtheorem{proposition}[theorem]{Proposition}
\newtheorem{corollary}[theorem]{Corollary}

\newtheorem{assumption}[theorem]{Assumption}

\theoremstyle{definition}
\newtheorem{definition}[theorem]{Definition}
\newtheorem{example}[theorem]{Example}

\theoremstyle{remark}
\newtheorem{remark}[theorem]{Remark}
\newtheorem{remarks}[theorem]{Remarks}

\newcommand{\kk}{\ensuremath{\Bbbk}} 

\newcommand{\CC}{\ensuremath{\mathbb{C}}}
\newcommand{\NN}{\ensuremath{\mathbb{N}}} 

\newcommand{\QQ}{\ensuremath{\mathbb{Q}}}

\newcommand{\ZZ}{\ensuremath{\mathbb{Z}}} 
\newcommand{\cM}{\mathcal{M}}
\newcommand{\cO}{\mathcal{O}}

\newcommand{\one}{\ensuremath{(\mathrm{i})}}
\newcommand{\two}{\ensuremath{(\mathrm{ii})}}

 \newcommand{\Amod}{A{\operatorname{-mod}}}
\newcommand{\ACmod}{A_C{\operatorname{-mod}}}

\DeclareMathOperator{\Cl}{Cl}

\DeclareMathOperator{\cok}{cok}

\newcommand{\Dfinb}{D^b_{\operatorname{fin}}}
\renewcommand{\div}{\operatorname{div}}

\DeclareMathOperator{\End}{End} 

\DeclareMathOperator{\Ext}{Ext}
\newcommand{\ghilb}{\ensuremath{G}\operatorname{-Hilb}}
\DeclareMathOperator{\git}{\!/\!\!/\!} 
\DeclareMathOperator{\GL}{GL} 
\DeclareMathOperator{\Gr}{Gr} 
\DeclareMathOperator{\Hom}{Hom}
\DeclareMathOperator{\head}{hd}
\DeclareMathOperator{\id}{id}
\DeclareMathOperator{\im}{Im}
\DeclareMathOperator{\inc}{inc}
\DeclareMathOperator{\Irr}{Irr}
\DeclareMathOperator{\Spe}{Sp}

\newcommand{\Kfin}{K_{\operatorname{fin}}}
\newcommand{\Knum}{K^{\operatorname{num}}}
\newcommand{\Knumc}{K^{\operatorname{num}}_c}

\DeclareMathOperator{\Ker}{Ker}
\DeclareMathOperator{\Coker}{Coker}
\DeclareMathOperator{\Pic}{Pic}
 
\DeclareMathOperator{\rank}{rk}
\DeclareMathOperator{\SO}{SO} 
\DeclareMathOperator{\Spec}{Spec}
\DeclareMathOperator{\SL}{SL} 
\DeclareMathOperator{\supp}{supp}
\DeclareMathOperator{\tail}{tl}
\newcommand{\vv}{\bm{\mathrm v}} 
\DeclareMathOperator{\module}{mod}
\DeclareMathOperator{\rk}{rk}

\begin{document}

\title{Multigraded linear series and recollement}

\author{Alastair Craw} 
\address{Department of Mathematical Sciences, 
University of Bath, 
Claverton Down, 
Bath BA2 7AY, 
United Kingdom.}
\email{a.craw@bath.ac.uk}
\urladdr{http://people.bath.ac.uk/ac886/}

\author{Yukari Ito}
\address{Graduate School of Mathematics, Nagoya University, Furocho, Chikusaku,
Nagoya 464-8602, Japan / Kavli Institute for the Physics, Mathematics of the Universe, The University of Tokyo,
5-1-5 Kashiwanoha, Kashiwa, Chiba, 277-8583, Japan}
\email{y-ito@math.nagoya-u.ac.jp / yukari.ito@ipmu.jp}
\urladdr{http://www.math.nagoya-u.ac.jp/\scriptsize{$\sim$}y-ito/}

\author{Joseph Karmazyn} 
\address{School of Mathematics and Statistics,
University of Sheffield,
Hicks Building,
Hounsfield Road,
Sheffield,
S3 7RH.}
\email{j.h.karmazyn@sheffield.ac.uk}
\urladdr{http://www.jhkarmazyn.staff.shef.ac.uk/}
\date{\today}


\begin{abstract}
 Given a scheme $Y$ equipped with a collection of globally generated vector bundles $E_1, \dots, E_n$, we study the universal morphism from $Y$ to a fine moduli space $\mathcal{M}(E)$ of cyclic modules over the endomorphism algebra of $E:=\mathcal{O}_Y\oplus E_1\oplus\cdots \oplus E_n$. This generalises the classical morphism to the linear series of a basepoint-free line bundle on a scheme. We describe the image of the morphism and present necessary and sufficient conditions for surjectivity in terms of a recollement of a module category. When the morphism is surjective, this gives a fine moduli space interpretation of the image, and as an application we show that for a small, finite subgroup $G\subset \GL(2,\kk)$,  every sub-minimal partial resolution of $\mathbb{A}^2_\kk/G$ is isomorphic to a fine moduli space $\mathcal{M}(E_C)$ where $E_C$ is a summand of the bundle $E$ defining the reconstruction algebra. We also consider applications to Gorenstein affine threefolds, where Reid's recipe sheds some light on the classes of algebra from which one can reconstruct a given crepant resolution. 
 \end{abstract}

\maketitle
\tableofcontents

\section{Introduction}
 The study of an algebraic variety in terms of the morphisms to the linear series of basepoint-free line bundles has always been a central tool in algebraic geometry. Here we extend this notion to the multigraded linear series of a collection of globally generated vector bundles on a scheme, thereby unifying several constructions from the literature (see \cite{CrawSmith08, CrawSpecial11} and \cite[Section~5]{Craw11}). Our primary goal is to provide new, geometrically significant moduli space descriptions of any given scheme, and we illustrate this in several families of examples. 
 
 \subsection*{Multigraded linear series}
 To be more explicit, let $Y$ be a scheme that is projective over an affine scheme of finite type over $\kk$, an algebraically closed field of characteristic zero. Given a collection $E_1, \dots, E_n$ of effective vector bundles on $Y$,  define $E:=\bigoplus_{0\leq i\leq n} E_i$ where $E_0$ is the trivial line bundle on $Y$. Let $A:=\End_Y(E)$ denote the endomorphism algebra and consider the dimension vector $\vv:=(v_i)$ given by $v_i:=\rk(E_i)$ for $0\leq i\leq n$. We define the \emph{multigraded linear series} of $E$ to be the fine moduli space $\mathcal{M}(E)$ of $0$-generated $A$-modules of dimension vector $\vv$ (see Definition~\ref{def:multigradedlinearseries}).  The universal family on $\mathcal{M}(E)$ is a vector bundle $T=\bigoplus_{0\leq i\leq n} T_i$ together with a $\kk$-algebra homomorphism $A\to \End(T)$, where $T_i$ is a tautological vector bundle of rank $v_i$ for $1\leq i\leq n$ and $T_0$ is the trivial line bundle. 
 
 Our first main result (see Theorem~\ref{thm:multigradedlinearseries}) generalises the classical morphism $\varphi_{\vert L\vert}\colon Y\to \vert L\vert$ to the linear series of a single basepoint-free line bundle $L$ on $Y$, or the morphism to a Grassmannian defined by a globally generated vector bundle on a projective variety:
 
\begin{theorem}
\label{thm:Intromultigradedlinearseries}
 If the vector bundles $E_1,\dots, E_n$ are globally generated, then there is a morphism $f\colon Y\to \cM(E)$ satisfying $E_i=f^*(T_i)$ for $0\leq i\leq n$ whose image is isomorphic to the image of the morphism $\varphi_{\vert L\vert}\colon Y\to \vert L\vert$ to the linear series of $L:=\bigotimes_{1\leq i\leq n} \det(E_i)^{\otimes j}$ for some $j>0$.
\end{theorem}
 
A statement similar to Theorem~\ref{thm:Intromultigradedlinearseries} holds if we replace $\mathcal{M}(E)$ by a product of Grassmannians over $\Gamma(\mathcal{O}_Y)$. However, the dimension of this product is higher than that of $\mathcal{M}(E)$ in general, and $f$ is almost never an isomorphism, i.e., $f$ does not provide a moduli space description of $Y$.  
 
 If the line bundle $\bigotimes_{1\leq i\leq n} \det(E_i)$ is ample, then after taking a multiple of a linearisation if necessary (see Remark~\ref{rem:ample}) the resulting universal morphism $f\colon Y\to \mathcal{M}(E)$ is a closed immersion and it is natural to ask whether $f$ is surjective, in which case $f$ presents $Y$ as the fine moduli space $\mathcal{M}(E)$. Even when $Y$ is isomorphic to $\mathcal{M}(E)$, one can sometimes gain more insight by deleting summands of $E$. Indeed, if $C\subseteq \{0,1,\dots, n\}$ is a subset containing $0$ then the subbundle $E_C:= \bigoplus_{i\in C} E_i$ of $E$ has the trivial line bundle $E_0$ as a summand, and Theorem~\ref{thm:Intromultigradedlinearseries} gives a universal morphism 
 \begin{equation}
 \label{eqn:introgC}
 g_C\colon \mathcal{M}(E)\longrightarrow \mathcal{M}(E_C)
 \end{equation}
 between multigraded linear series which can lead to a more \emph{geometrically significant} moduli space description of $Y$. 

A moduli construction determined by a tilting bundle is of clear geometric significance. For a smooth variety $Y$ admitting a tilting bundle $E$, Bergman-Proudfoot~\cite[Theorem~2.4]{BP08} showed that $f$ is an isomorphism onto a connected component of $\cM(E)$. The goal of this paper is to establish several situations in which $f$ is an isomorphism onto $\cM(E)$ itself, thereby giving a moduli space description of $Y$. We do not assume that $Y$ is smooth (it is singular in Theorem~\ref{thm:introsurjectioncontraction}\two), and while $E$ may be a tilting bundle, we do not demand this much; after all, one does not require every indecomposable summand of Beilinson's tilting bundle in order to reconstruct $\mathbb{P}^n$.
  
 \subsection*{The Special McKay correspondence}
 Our second main result illustrates this phenomenon. Let $G\subset \GL(2,\kk)$ be a finite subgroup without pseudo-reflections, write $\Irr(G)$ for the set of isomorphism classes of irreducible representations of $G$, and let $Y$ denote the minimal resolution of $\mathbb{A}^2_\kk/G$. Generalising the work of Ito--Nakamura~\cite{ItoNakamura99} for a finite subgroup of $\SL(2,\kk)$, Kidoh~\cite{Kidoh01} and Ishii~\cite{Ishii02} proved that $Y$ is isomorphic to the $G$-Hilbert scheme, that is, the fine moduli space of $G$-equivariant coherent sheaves of the form $\mathcal{O}_Z$ for subschemes $Z\subset \mathbb{A}^2_\kk$ such that $\Gamma(\mathcal{O}_Z)$ is isomorphic to the regular representation of $G$. If we write 
 \[
 T:=\bigoplus_{\rho\in \Irr(G)} T_\rho^{\oplus\dim(\rho)}
 \]
 for the tautological bundle on the $G$-Hilbert scheme, then since $\End_Y(T)$ is isomorphic to the skew group algebra (see Lemma \ref{lem:TautologicalBundleEndomorphismAlgebra}), it follows that the minimal resolution $Y\cong \ghilb$ is isomorphic to the multigraded linear series $\mathcal{M}(T)$. When $G$ is a finite subgroup of $\SL(2,\kk)$, $T$ is a tilting bundle on $Y$ by work of Kapranov--Vasserot~\cite{KapranovVasserot00}, so $Y$ is derived equivalent to the category of modules over the endomorphism algebra of $T$. However, this is false in general; put simply, the $G$-Hilbert scheme is the wrong moduli description of the minimal resolution of $\mathbb{A}^2_\kk/G$ unless $G$ is a finite subgroup of $\SL(2,\kk)$. 
 
 A more natural moduli space description of $Y$ comes from the Special McKay correspondence of the finite subgroup $G\subset \GL(2,\kk)$. For the set $\Spe(G):= \{ \rho \in \Irr(G) \mid H^1(T_{\rho}^{\vee})=0 \}$ of special representations, it follows from Van den Bergh~\cite{VdB04Duke} that the \emph{reconstruction bundle}
 \[
 E:=\bigoplus_{\rho\in \Spe(G)} T_\rho
 \]
 is a tilting bundle on $Y$, so $Y$ is derived equivalent to the module category of the endomorphism algebra of $E$, that is, to the category of modules over the reconstruction algebra studied by Wemyss~\cite{WemyssGL211}. Our second main result (see Proposition~\ref{prop:reconstructionisomorphism} and Theorem~\ref{thm:surjectioncontraction}) shows that $E$ contains enough information to reconstruct $Y$, and hence provides a moduli space description that trumps the $G$-Hilbert scheme in general:
 
\begin{theorem} 
\label{thm:introsurjectioncontraction}
 Let $G\subset \GL(2,\kk)$ be a finite subgroup without pseudo-reflections. Then:
 \begin{enumerate}
 \item[\one] the minimal resolution $Y$ of $\mathbb{A}^2_\kk/G$ is isomorphic to the multigraded linear series $\mathcal{M}(E)$ of the reconstruction bundle; and
 \item[\two] for any partial resolution $Y^\prime$ such that the minimal resolution $Y\to \mathbb{A}^2_\kk/G$ factors via $Y^\prime$, there is a summand $E_C\subseteq E$ such that $Y^\prime$ is isomorphic to $\mathcal{M}(E_{C})$.
 \end{enumerate}
\end{theorem}

 In other words, for a finite subgroup $G\subset \GL(2,\kk)$, the minimal resolution of $\mathbb{A}^2_\kk/G$ can be obtained directly from the special representations. The statement of part \one\ is due originally to Karmazyn~\cite[Corollary~5.4.5]{Karmazyn14}, while an analogue of part \two\ in the complete local setting can be deduced by combining Iyama--Kalck--Wemyss--Yang~\cite[Theorem~4.6]{KIWY} with \cite[Corollary~5.2.5]{Karmazyn14}. Note however that our approach is completely different in each case, and is closer in spirit to the geometric construction of the Special McKay correspondence for cyclic subgroups of $\GL(2,\kk)$ given by Craw~\cite{CrawSpecial11}.

 \subsection*{Main tools}
 The key to the proof of Theorem~\ref{thm:introsurjectioncontraction} is a homological criterion to decide when the morphism $g_C$ from \eqref{eqn:introgC} is surjective. In this situation, any subset $C\subseteq \{0,1,\dots, n\}$ containing $0$ determines a subbundle $E_C$ of $E$, and the module categories of the algebras $A:=\End_Y(E)$ and $A_C:=\End_Y(E_C)$ are linked by a recollement (see Section \ref{Section:Thecorneringcategoryandrecollement}). In particular, there is an exact functor $j^*\colon \Amod \rightarrow \ACmod$ with left adjoint $j_!\colon \ACmod \rightarrow \Amod$. These functors capture information about the morphism $g_C\colon\mathcal{M}(E) \rightarrow \mathcal{M}(E_C)$ from \eqref{eqn:introgC}: closed points $y \in \mathcal{M}(E)$ and $x \in \mathcal{M}(E_C)$ correspond to $0$-generated modules $M_y \in \Amod$ and $N_x \in \ACmod$ of dimension vectors $\vv:=(v_i)_{0\leq i\leq n}$ and $\vv_C:=(v_i)_{i\in C}$ respectively, and 
\[
g_C(y)=x \iff j^*M_y=N_x.
\]
Since the functor $j_!$ lifts $0$-generated $A_C$-modules to $0$-generated $A$-modules, the question of whether $x$ lies in the image of $g_C$ reduces to the following (see Proposition \ref{prop:submoduleSurjection}):

\begin{proposition}
 The morphism $g_C\colon \mathcal{M}(E)\to \mathcal{M}(E_C)$ is surjective iff for each $x\in \mathcal{M}(E_C)$, the $A$-module $j_!(N_x)$ admits a surjective map onto an $A$-module of dimension vector $\vv$.
\end{proposition}

A second key ingredient is that a derived equivalence
\[
\Psi(-):=E^\vee\otimes_A - \colon D^b(A)\longrightarrow D^b(Y)
\]
induces an isomorphism between the lattice of dimension vectors for $A$ and the numerical Grothendieck group for compact support  $\Knumc(Y)$, introduced by Bayer--Craw--Zhang~\cite{BayerCrawZhang16} (see Appendix~\ref{sec:appendix}). In particular, understanding the class of the object $\Psi(j_!(N_x))$ in $\Knumc(Y)$ reveals the dimension vector of the $A$-module $j_!(N_x)$ for each closed point $x\in \mathcal{M}(E_C)$, and this provides a tool to help determine whether the above the homological criterion applies. More explicitly, we prove the following result (see Theorem~\ref{thm:recollement2}):

\begin{theorem}
\label{thm:introrecollement2}
Suppose that for each $x\in \mathcal{M}(E_C)$, the class $[\Psi(j_!(N_x))]\in \Knumc(Y)$ can be written as a positive combination of the classes of sheaves on $Y$. Then $g_C$ is surjective. 
\end{theorem}

\subsection*{Examples from NCCRs in dimension three}
The morphisms $Y\rightarrow \mathbb{A}^2/G$ from Theorem~\ref{thm:introsurjectioncontraction} all have fibres of dimension at most one, but our methods apply without this assumption. To illustrate this, we also study crepant resolutions of Gorenstein affine threefolds. For any such singularity $X$, Van den Bergh's construction~\cite{VdB04} of an NCCR (satisfying Assumption~\ref{ass:KrullSchmidt}) produces a crepant resolution of $X$ as a fine moduli space of stable representations for an algebra $A$ (see Proposition~\ref{prop:connected}). The choice of $0$-generated stability condition chooses a particular crepant resolution $Y$ and a globally generated bundle $T$ on $Y$. In this case, $Y$ is isomorphic to the multigraded linear series $\mathcal{M}(T)$; since $T$ is a tilting bundle, this moduli construction is certainly geometrically significant. 

Nevertheless, motivated by work of Takahashi~\cite{Takahashi11}, we ask whether one can reconstruct $Y$ using only a proper summand of $T$ (in general, none of the indecomposable summands of $T$ is ample).  To state the result, we say that a vertex $i\in Q_0=\{0,1,\dots, n\}$ is \emph{essential} if there is a $0$-generated $A$-module of dimension vector $\vv$ that contains the vertex simple $A$-module $S_i$ in its socle. The following result combines Propositions~\ref{prop:SekiyaTakahashi} and~\ref{prop:inessentials}:
 
 \begin{proposition}
 \label{prop:introTakahashi}
 Let $\mathcal{M}(T)$ be the crepant resolution of the Gorenstein, affine toric threefold $X$ picked out by the choice of $0$-generated stability condition as above. Then:
 \begin{enumerate}
 \item[\one] for any subset $C\subseteq Q_0$ containing $0$, the image of $g_C\colon \mathcal{M}(T)\to\mathcal{M}(T_C)$ is an irreducible component of $\mathcal{M}(T_C)$; and
 \item[\two] if $C$ is the union of $\{0\}$ with the set of essential vertices, then $g_C$ is an isomorphism onto its image.
 \end{enumerate}
\end{proposition}
 
 \noindent Here, part \two\ generalises the result of Takahashi~\cite{Takahashi11} beyond the case where $X$ is an abelian quotient. Example~\ref{exa:1/3(1,1,1)} shows that Proposition~\ref{prop:introTakahashi} is optimal: in general $\mathcal{M}(T_C)$ is reducible.
 
 We conclude with several examples that are orthogonal in spirit to Proposition~\ref{prop:introTakahashi}\two, with a view to strengthening the statement to an isomorphism $\mathcal{M}(T)\cong \mathcal{M}(T_C)$. Rather than keep the summands of $T$ corresponding to essential vertices, instead we use Reid's recipe as a guide to help us choose which essential vertices to \emph{remove}.  For an essential vertex $i\in Q_0$, derived Reid's recipe~\cite{CautisLogvinenko09,Logvinenko10,CautisCrawLogvinenko12,BCQV} proves that the image under the derived equivalence $\Psi$ of the vertex simple $A$-module $S_i$ is a sheaf. This gives enough information to compute the class of $\Psi(j_!(N_x))$ in $\Knumc(Y)$, and under a dimension condition (see Corollary~\ref{cor:surjective}), we can apply Theorem~\ref{thm:introrecollement2} to deduce that $g_C\colon \mathcal{M}(T)\to\mathcal{M}(T_C)$ is an isomorphism.  We showcase this construction in Examples~\ref{ex:1/6(1,2,3)} and~\ref{ex:DP6}, and in the latter example we also show that the dimension condition can fail if we remove two indecomposable summands of $T$ corresponding to essential vertices that mark the same surface by Reid's recipe. Put more geometrically, surjectivity can fail if the summands of $T_C$ do not generate the Picard group of $\mathcal{M}(T)$. 

\subsection*{Notation} 
Let $\kk$ be an algebraically closed field of characteristic zero. For any quasiprojective $\kk$-scheme $Y$ and $\kk$-algebra $A$, we write $D^b(Y)$ and $D^b(A)$ for the bounded derived categories of coherent sheaves on $Y$ and finitely generated left $A$-modules respectively. A vector bundle is a locally-free sheaf of finite rank.  

\subsection*{Acknowledgements} The first author thanks Stefan Schr\"{o}er for a stimulating discussion. We thank the anonymous referees for pointing out an error in an earlier version of this paper and for helpful comments. The first and third authors were supported by EPSRC grants EP/J019410/1 and EP/M017516/1 respectively, and the second author was supported by JSPS Grant-in-Aid (C) No. 23540045. 
 
\section{Multigraded linear series} \label{Section:MultigradedLinearSeries}
An associative $\kk$-algebra $A$ that is presented in the form $A\cong\kk Q/I$ for some finite connected quiver $Q$ and two-sided ideal $I\subset \kk Q$ determines a choice of idempotents $e_i\in A$, one for each vertex $i\in Q_0$. Let  $\ZZ^{Q_0}$ denote the free abelian group generated by the vertex set of $Q$. A dimension vector $\vv=(v_i)\in \NN^{Q_0}$ determines the rational vector space 
\[
\Theta_{\vv}:= \vv^\perp=\Bigg\{\theta=(\theta_i)\in \Hom(\ZZ^{Q_0},\QQ) \mid \sum_{i\in Q_0} \theta_iv_i= 0\Bigg\}
\]
 of stability parameters for $A$-modules of dimension vector $\vv$. For $\theta\in \Theta_{\vv}$, an $A$-module $M$ of dimension vector $\vv$ is \emph{$\theta$-semistable} if $\theta(N) \geq 0$ for every nonzero proper $A$-submodule $N$ of $M$. The notion of $\theta$-stability is defined by replacing  $\geq$ with $>$, and we say $\theta\in \Theta_{\vv}$ is \emph{generic} if every $\theta$-semistable $A$-module is $\theta$-stable. There is a wall and chamber decomposition on $\Theta_{\vv}$,  where two generic parameters $\theta, \theta^\prime\in \Theta_{\vv}$ lie in the same chamber if and only if the notions of $\theta$-stability and $\theta^\prime$-stability coincide.
 
 When $\vv$ is indivisible and $\theta\in \Theta_{\vv}$ is generic, King~\cite{King94} constructs the fine moduli space $\mathcal{M}(A,\vv,\theta)$ of isomorphism classes of $\theta$-stable $A$-modules of dimension vector $\vv$. The universal family on $\mathcal{M}(A,\vv,\theta)$ is a tautological vector bundle
 \[
 T=\bigoplus_{i\in Q_0} T_i
 \]
 satisfying $\rank(T_i)=v_i$ for $i\in Q_0$, together with a $\kk$-algebra homomorphism $A\to \End(T)$, such that the fibre of $T$ at any closed point of $\mathcal{M}(A,\vv,\theta)$ is the corresponding $\theta$-stable $A$-module of dimension vector $\vv$. In fact, $T$ is defined only up to tensor product by an invertible sheaf, but we remove this ambiguity by choosing once and for all a vertex of the quiver that we denote $0\in Q_0$ and working only with dimension vectors $\vv$ satisfying $v_0=1$; we normalise $T$ by fixing $T_0$ to be the trivial line bundle. Let $L_\theta := \bigotimes_{i\in Q_0} \det(T_i)^{\theta_i}$ denote the ample bundle on $\mathcal{M}(A,\vv,\theta)$ induced by the GIT construction.
\medskip

 These moduli spaces arise naturally in geometry as follows. Let $R$ be a finitely generated $\kk$-algebra, let $Y$ be a projective $R$-scheme and let $E_1,\dots, E_n$ be nontrivial, effective vector bundles on $Y$. In addition, for $E_0:=\cO_Y$, write $A:=\End(\bigoplus_{0\leq i\leq n}E_i)$ for the endomorphism algebra. This decomposition of $E:=\bigoplus_{0\leq i\leq n}E_i$ gives a complete set of orthogonal idempotents $e_i=\id\in \End(E_i)$ of $A$ such that $1=\sum_{0\leq i\leq n} e_i$. Then $A$ admits a presentation
 \begin{equation}
 \label{eqn:quiveralgebra}
 A\cong \kk Q/I
 \end{equation}
 such that the vertex set is $Q_0=\{0,1,\dots,n\}$. Indeed, introduce a set of loops at each vertex corresponding to a finite set of $\kk$-algebra generators of $R$, and 
 for $0\leq i,j\leq n$ we introduce arrows from $i$ to $j$ corresponding to a finite generating set for $\Hom(E_i,E_j)$ as an $R$-module. This determines a surjective $\kk$-algebra homomorphism $\kk Q\to \End(E)$ with kernel $I$. 
 
 \begin{remark}
 \label{rem:idempotents}
 The ideal $I\subset \kk Q$ constructed in this way need not be admissible, and relations may even involve idempotents. For example, if $E_i\cong E_j$ then the isomorphisms correspond to relations of the form $aa^\prime-e_i, a^\prime a-e_j\in I$ for some $a, a^\prime \in \kk Q$.
 \end{remark}

 The dimension vector $\vv=(v_i)$ defined by setting $v_i:=\rank(E_i)$ for $0\leq i\leq n$ is indivisible, and $E$ is a flat family of $A$-modules of dimension vector $\vv$. If there exists generic $\theta\in \Theta_{\vv}$ such that for each closed point $y\in Y$, the fibre $E_y$ of $E$ over $y$ is $\theta$-stable, then the universal property of $\cM(A,\vv,\theta)$ determines a morphism 
\[
f\colon Y\longrightarrow \cM(A,\vv,\theta)
\]
satisfying $E_i=f^*(T_i)$ for all $i\in Q_0$; note that $\cM(A,\vv,\theta)$ depends on $E$ and the GIT chamber containing $\theta$, while $f$ depends on both $E$ and $\theta$. The following result is well known to experts (compare Bergman--Proudfoot~\cite[Theorem~2.4]{BP08}), but we were unable to find the statement at the level of generality we require so for convenience we provide a proof.

\begin{lemma}
\label{lem:universal}
 Given a vector bundle $E$ on $Y$, suppose there exists a generic $\theta=(\theta_i)\in \Theta_{\vv}$ such that $E$ is a flat family of $\theta$-stable $A$-modules of dimension vector $\vv$. 
 \begin{enumerate}
 \item[\one] The pullback via the universal morphism $f\colon Y\to \cM(A,\vv,\theta)$ of the ample bundle $L_\theta$ induced by the GIT construction is the line bundle $\bigotimes_{0\leq i\leq n} \det(E_i)^{\otimes\theta_i}$ on $Y$; and
 \item[\two] if $L_\theta$ is very ample (which holds after replacing $\theta$ by a positive multiple if necessary), then the image of $f\colon Y\to \cM(A,\vv,\theta)$ is isomorphic to the image of the morphism from $Y$ to the linear series of the globally generated line bundle $\bigotimes_{0\leq i\leq n} \det(E_i)^{\otimes\theta_i}$.
 \end{enumerate}
 \end{lemma}
 
\begin{proof}
 Since pullback commutes with tensor operations on $T_i$, the universal property of the morphism $f$ gives $\det(E_i)^{\otimes\theta_i} = \det(f^*(T_i))^{\otimes\theta_i} \cong f^*(\det(T_i)^{\otimes \theta_i})$ for $0\leq i\leq n$, and hence 
 \begin{equation}
 \label{eqn:pullbackuniversal}
 \bigotimes_{0\leq i\leq n} \det(E_i)^{\otimes\theta_i}\cong f^*\left(\bigotimes_{0\leq i\leq n} \det(T_i)^{\otimes\theta_i}\right) = f^*(L_\theta)
 \end{equation}
 which proves \one. For \two, let $g\colon \cM(A,\vv,\theta)\to \vert L_\theta\vert$ denote the closed immersion to the linear series of $L_\theta$. Equation \eqref{eqn:pullbackuniversal} gives $L:=(g\circ f)^*(\cO(1)) = \bigotimes_{0\leq i\leq n} \det(E_i)^{\otimes\theta_i}$. It follows that $g\circ f$ coincides with the classical morphism $\varphi_{\vert L\vert}\colon Y\to \vert L\vert$ to the linear series of $L$. In particular, $\bigotimes_{0\leq i\leq n} \det(E_i)^{\otimes\theta_i}$ is globally generated. Moreover, since $g$ is a closed immersion, the image of $f$ is isomorphic via $g$ to the image of $\varphi_{\vert L\vert}$.
\end{proof}
  
 The problem with Lemma~\ref{lem:universal} is that it is a difficult problem in general to find a suitable parameter $\theta\in \Theta_{\vv}$ for which a given vector bundle $E$ defines a flat family of $\theta$-stable $A$-modules.  
 
 Here we highlight a special situation where this problem has a simple solution. It's easy to see that any stability parameter  $\theta=(\theta_i)\in \Theta_{\vv}$ satisfying $\theta_i>0$ for all $i\neq 0$ is generic, so there is a GIT chamber $\Theta_{\vv}^+\subset \Theta_{\vv}$ containing all such stability parameters. Given an $A$-module $M$ that is $\theta$-stable for $\theta\in \Theta^+_{\vv}$, it follows directly from the definition that there exists a surjective $A$-module homomorphism $Ae_0\to M$. More generally, we say that an $A$-module $M$ is \emph{$0$-generated} if there exists a surjective $A$-module homomorphism $Ae_0\to M$. It is sometimes advantageous to use this latter notion because it is well-defined without having to make explicit reference to a dimension vector $\vv$.
 
\begin{proposition}
\label{prop:0-generated}
 Let $E_1, \dots, E_n$ be vector bundles on $Y$ and set $E_0=\cO_Y$. Then $\bigoplus_{0\leq i\leq n} E_i$ is a flat family of $0$-generated $A$-modules if and only if $E_i$ is globally generated for all $1\leq i\leq n$.  
 \end{proposition}
\begin{proof}
The bundle $E:=\bigoplus_{0\leq i\leq n} E_i$ on $Y$ is a flat family of $A$-modules of dimension vector $\vv=(\rank(E_i))_{0\leq i\leq n}$, and for any closed point $y\in Y$, the $A$-module structure in the fibre $E_y$ of $E$ over $y$ is obtained by evaluating all homomorphisms between the bundles $E_i$ (for $0\leq i\leq n$) at $y$. Choose $\theta\in \Theta_{\vv}^+$ satisfying $\theta_i>0$ for $i\neq 0$. Since $\theta_0=-\sum_{1\leq i\leq n} v_i\theta_i$, an $A$-submodule $W$ of $E_y$ is destabilising if and only if $\dim(W_0)=1$ and there exists $i>0$ such that the sum of all maps from $W_0$ to $W_i$ determined by paths in the quiver from $0$ to $i$ is not surjective. In particular, $E_y$ is $\theta$-unstable for some $y\in Y$ if and only if there exists $i>0$ such that $E_i$ cannot be written as the quotient of $\mathcal{O}_Y^{\oplus k}$ for some $k\in \NN$; equivalently, $E_y$ is $\theta$-stable for all $y\in Y$ if and only if $E_i$ is globally generated for $1\leq i\leq n$. 
\end{proof}

 \begin{corollary}
 \label{cor:ggtautbundles}
 For any $\theta\in \Theta_{\vv}^+$, the locally-free sheaf $T_i$ on $\cM(A,\vv,\theta)$ is globally generated for all $0\leq i\leq n$.
 \end{corollary}
 \begin{proof}
 The sheaf $\bigoplus_{0\leq i\leq n} T_i$ is a flat family of $\theta$-stable $A$-modules of dimension vector $\vv$ on $\cM(A,\vv,\theta)$. The result follows from Proposition~\ref{prop:0-generated} and the fact that $T_0\cong\cO_{\cM(A,\vv,\theta)}$.
 \end{proof}

 \begin{definition}
 \label{def:multigradedlinearseries}
 Let $E_1,\dots, E_n$ be globally generated vector bundles on $Y$. For $E_0:=\cO_Y$, write $E=\bigoplus_{0\leq i\leq n} E_i$, and define  $A:=\End_Y(E)$ and  $\vv=(\rank(E_i))_{0\leq i\leq n}$. For $\theta' \in \Theta_{\vv}^+$ satisfying $\theta^\prime_i = 1$ for all $i>0$, define $j>0$ to be the smallest integer such that the ample line bundle $L_{\theta'}^{\otimes j}$ on $\cM(A, \vv, \theta')$ is very ample.  The \emph{multigraded linear series} of $E$ is the fine moduli space
 \[
 \cM(E):= \cM(A, \vv, \theta)
 \]
 of $\theta$-stable $A$-modules of dimension vector $\vv$, where $\theta:=j\theta'$.  
 \end{definition}
 
 Corollary~\ref{cor:ggtautbundles} implies that each direct summand of the tautological bundle $T=\bigoplus_{0\leq i\leq n} T_i$ on $\cM(E)$ is globally generated. Moreover, $L_\theta = \bigotimes_{1\leq i\leq n} \det(T_i)^{\otimes j}$ is very ample. To justify the terminology `multigraded linear series', we present the following result (compare Example~\ref{exa:linearseries}).
  
\begin{theorem}
\label{thm:multigradedlinearseries}
 Let $E_1,\dots, E_n$ be globally generated vector bundles on $Y$. There is a morphism $f\colon Y\to \cM(E)$ satisfying $E_i=f^*(T_i)$ for $1\leq i\leq n$ whose image is isomorphic to the image of the morphism $\varphi_{\vert L\vert}\colon Y\to \vert L\vert$ to the linear series of $L:=\bigotimes_{1\leq i\leq n} \det(E_i)^{\otimes j}$ for some $j>0$.
\end{theorem}
\begin{proof}
Apply Proposition~\ref{prop:0-generated} and Lemma~\ref{lem:universal}\two\ to the parameter $\theta = j\theta^\prime\in \Theta_{\vv}^+$ from Definition~\ref{def:multigradedlinearseries}, noting that $E_0=\cO_Y$. 
\end{proof}

\begin{example}[Linear series of higher rank]
\label{exa:linearseries}
When $Y$ is projective and $E_1$ has rank $r$, then $\cM(A,\vv,\theta)$ is isomorphic to the Grassmannian $\Gr(H^0(E_1),r)$ of rank $r$ quotients of $H^0(E_1)$. The ample bundle $\cO(1)$ on $\Gr(H^0(E_1),r)$ is very ample, so we may take $j=1$ in Definition~\ref{def:multigradedlinearseries} and Theorem~\ref{thm:multigradedlinearseries}. Therefore $f$ coincides with the morphism $\varphi_{\vert E_1\vert}$ to the linear series of higher rank that recovers $E_1$ as the pullback of the tautological quotient bundle of rank $r$; see Mukai~\cite[Section~3]{Mukai2010}. When $r=1$, this is the classical linear series of a basepoint-free line bundle.
\end{example}

\begin{remark}
\label{rem:ample}
If the globally generated line bundle $L:=\bigotimes_{1\leq i\leq n} \det(E_i)^{\otimes j}$ from Theorem~\ref{thm:multigradedlinearseries} is ample but not very ample, choose $k>0$ such that $L^{\otimes k}$ is very ample. After replacing $\theta$ by $k\theta$, the proof of Theorem~\ref{thm:multigradedlinearseries} shows that the resulting universal morphism 
\[
f\colon Y\to \cM(A,\vv,k\theta)
\]
 is a closed immersion. Since $\cM(A,\vv,k\theta)\cong \cM(E)$, we will hereafter simply write $f\colon Y\to \cM(E)$ for the closed immersion given by the composition whenever $L$ is ample but not very ample. 
 \end{remark}

\begin{example}[Quiver flag varieties]
 Let $Q$ be a finite, acyclic, connected quiver with a unique source  denoted $0\in Q_0$, and let $\vv=(v_i)\in \NN^{Q_0}$ be a dimension vector satisfying $v_0=1$. For $\theta\in \Theta_{\vv}^+$, the fine moduli space $\mathcal{M}(\kk Q, \vv, \theta)$ is called a \emph{quiver flag variety} \cite{Craw11}, and for $i\in Q_0$, the tautological bundle $T_i$ of rank $v_i$ is globally generated by Corollary~\ref{cor:ggtautbundles}.
 Every such variety is an iterative Grassmann-bundle \cite[Theorem~3.3]{Craw11}, and $T_i$ is simply the pullback of the tautological quotient bundle on one of the Grassmann bundles in the tower. 
 
 We claim that $\mathcal{M}(\kk Q, \vv, \theta)$ is the multigraded linear series $\mathcal{M}(T)$ associated to $T=\bigoplus_{i\in Q_0}T_i$. Since $v_i=\rk(T_i)$ for $i\in Q_0$, it suffices to show that the tautological $\kk$-algebra homomorphism 
 \[
 h\colon \kk Q\to \End(T)
 \]
 is an isomorphism. The proof of \cite[Lemma~5.3]{Craw11} establishes that $h$ is surjective, so suppose $\sum_{1\leq \alpha\leq k} c_\alpha p_\alpha\in \Ker(h)$, where paths $p_1, \dots, p_k$ in $Q$ have common tail at $i\in Q_0$ and common head at $j\in Q_0$. For any path $p$ in $Q$ from vertex $0$ to $i$, the sum $\sum_{1\leq \alpha\leq k} c_\alpha p_\alpha p$ lies in the kernel of the map $e_0(\kk Q)e_i\to \Hom(T_0,T_i)\cong \Gamma(T_i)$ of $\kk$-vector spaces induced by $h$. However, this map is an isomorphism \cite[Corollary~3.5\two]{Craw11}, so $h$ is injective as required.  
  
 We may now apply Theorem~\ref{thm:multigradedlinearseries} to the determinants of the tautological bundles. Indeed, $\det(T_i)$ is globally generated for $i\in Q_0$ and $L=\bigotimes_{i\in Q_0} \det(T_i)$ is ample by \cite[Lemma~3.7]{Craw11},  so for $E=\bigoplus_{i\in Q_0} \det(T_i)$, the morphism $f\colon \mathcal{M}(T)\to \mathcal{M}(E)$ is a closed immersion.
 \end{example}
 
\begin{remark}
Multigraded linear series were introduced by Craw--Smith~\cite{CrawSmith08} when each $E_i$ has rank one and $Y$ is a projective toric variety.
\begin{enumerate}
\item The construction of \cite{CrawSmith08} is phrased in terms of a quiver with relations that gives a presentation $A\cong \kk Q/I$, leading to an explicit GIT description of the image of $f$ in a quiver flag variety. However, we now assume only that $Y$ is projective over an affine, and in this generality no such natural presentation exists, leading us to place greater emphasis on $A$ rather than on a quiver. While we sacrifice an explicit description of the image of $f$, Theorem~\ref{thm:multigradedlinearseries} nevertheless determines the image of $f$ up to isomorphism.
 \item The observation in Theorem~\ref{thm:multigradedlinearseries} that the image of $f$ is determined by $\bigotimes_{1\leq i\leq n} \det(E_i)^{\otimes j}$ for some $j> 0$ (see also Remark~\ref{rem:ample}) renders redundant the assumption from \cite[Corollary~4.10]{CrawSmith08} (and hence from \cite[Proposition~5.5]{Craw11} and Prabhu-Naik~\cite[Proposition~5.5]{PrabhuNaik15}) that a map obtained by multiplication of global sections is surjective. 
\end{enumerate}
\end{remark}

\section{The cornering category and recollement} \label{Section:Thecorneringcategoryandrecollement}
 In this section we use Theorem~\ref{thm:multigradedlinearseries} repeatedly to produce a compatible family of morphisms between different multigraded linear series.  We then introduce a homological criterion that is sufficient to guarantee that any of these morphisms is surjective.
 
 We continue to assume that $R$ is a finitely generated $\kk$-algebra, $Y$ is a projective $R$-scheme and that $E_1,\dots, E_n$ are nontrivial  globally generated vector bundles on $Y$. For $E_0=\cO_Y$,write $E:=\bigoplus_{0\leq i\leq n} E_i$ and define $A:= \End(E)$. To produce new endomorphism algebras from $A$, let $C\subseteq \{0,1,\dots,n\}$ be any subset containing $\{0\}$. Define both the idempotent  $e_C:= \sum_{i\in C} e_i$ of $A$ and the $\kk$-algebra $A_C:= e_CAe_C$. Since $A\cong\End(E)$, the locally-free sheaf $E_C:= \bigoplus_{i\in C} E_i$ on $Y$ satisfies
  \[
  A_C\cong  \End(E_C).
  \]
   The process of passing from $A$ to $A_C$ is called \emph{cornering} the algebra $A$. 
\begin{remark} \label{rem:cornering}
This use of the term \emph{cornering} comes from representation theory, where the basic example is cornering the matrix algebra $\End_{\kk}(\kk \oplus \kk) \cong\textnormal{Mat}_{2 \times 2}(\kk)$ by a nontrivial idempotent to produce a subalgebra of matrices that have nonzero entries only in one particular ``corner". 

This is distinct from the construction of Ishii--Ueda~\cite{IshiiUeda09} for dimer models, which is a process linking the removal of a  corner in a lattice polygon to the universal localisation of certain arrows in a quiver in order to determine an open subset in an associated moduli space.
\end{remark}

For $\vv_C:=(\rk(E_i))_{i\in C}$ and for any $0$-generated stability parameter $\theta_C\in \Theta_{\vv_{C}}^+$, Theorem~\ref{thm:multigradedlinearseries} gives a morphism 
 \begin{equation}
 \label{eqn:cornermorphism}
 f_C\colon Y\longrightarrow \cM(E_C)=\cM(A_C,\vv_C,\theta_C).
 \end{equation}
 
 Before studying these morphisms in detail, we record the fact that if we're interested only in the scheme underlying a multigraded linear series, we may assume without loss of generality that the globally generated vector bundles $E_1,\dots, E_n$  are pairwise non-isomorphic:  
 
 \begin{lemma}
 \label{lem:Morita}
 For $a_1,\dots, a_n\geq 1$, we have $\cM(\bigoplus_{0\leq i\leq n}E_i)\cong \cM(\bigoplus_{0\leq i\leq n}E_i^{\oplus a_i})$.
 \end{lemma}
\begin{proof}
The endomorphism algebra of $E^\prime:=\bigoplus_{0\leq i\leq n}E_i$ is Morita equivalent to the endomorphism algebra of $E:= \bigoplus_{0\leq i\leq n}E_i^{\oplus a_i}$. By increasing the multiplicities of the summands in the tautological bundle on $\cM(E^\prime)$, we obtain a flat family $\bigoplus_{0\leq i\leq n} T_i^{\oplus a_i}$ of $0$-generated $\End(E)$-modules of dimension vector $\vv = (a_i\rk(E_i))$, so there is a morphism $f\colon \cM(E^\prime)\to \cM(E)$. For the other direction, we have $E^\prime = E_C$ for some cornering subset $C$, so \eqref{eqn:cornermorphism} defines a morphism $f_C\colon \cM(E)\to \cM(E^\prime)$. Universality ensures that these morphisms are mutually inverse.
\end{proof}

 Our next result encodes the fact that the morphisms \eqref{eqn:cornermorphism} are compatible as we vary the choice of the subset $C$. To state the result, regard the poset of subsets of $\{0,1,\dots,n\}$ that contain $\{0\}$ as a category $\mathscr{C}$ in which the morphisms are the set-theoretic inclusion maps between subsets. 
 
 \begin{proposition}
 \label{prop:functor}
  Let $E_1,\dots, E_n$ be globally generated vector bundles on $Y$ and set $R:= \Gamma(\cO_Y)$. There is a contravariant functor from $\mathscr{C}$ to the category of $R$-schemes that sends a set $C$ to the multigraded linear series $\cM(E_C)$.
 \end{proposition}

\begin{proof}
  Given subsets $C^\prime \subseteq C\subseteq \{0,1,\dots, n\}$ that both contain $\{0\}$, we first construct a morphism $f_{C,C^\prime}\colon \cM(E_C)\to \cM(E_{C^\prime})$ that fits into a commutative diagram 
\begin{equation}
\label{eqn:diagfunctor}
\xymatrix{
 & Y\ar[dl]_{f_C}\ar[dr]^{f_{C^\prime}} & \\
 \cM(E_C)\ar[rr]^{f_{C,C^\prime}} & & \cM(E_{C^\prime}).
 }
\end{equation}
 To avoid a proliferation of indices, we write $T=\bigoplus_{i\in C}T_i$ and $T^\prime=\bigoplus_{i\in C^\prime}T^\prime_i$ for the tautological bundles on $\cM(E_C)$ and $\cM(E_{C^\prime})$ respectively, where $\rk(T^\prime_i) = \rk(T_i)=\rk(E_i)$ for all $i\in C^\prime$. Since $T$ is a flat family of $0$-generated $A_C$-modules on $\cM(E_C)$, it follows that $\bigoplus_{i\in C^\prime}T_i$ is a flat family of $0$-generated $A_{C^\prime}$-modules on $\cM(E_C)$. The universal property of $\cM(E_{C^\prime})$ as in Lemma~\ref{lem:universal} defines a morphism $f_{C,C^\prime}\colon \cM(E_C)\to \cM(E_{C^\prime})$ satisfying $f_{C,C^\prime}^*(T^\prime_i) = T_i$ for $i\in C^\prime$. The morphism $f_C$ is also universal, so 
 \[
 (f_{C,C^\prime}\circ f_C)^*(T_i^\prime) = f_C^*(T_i) = E_i
 \]
 for all $i\in C^\prime$. This property characterises $f_{C^\prime}$, so diagram \eqref{eqn:diagfunctor} commutes as required. That the assignment sending the inclusion $C^\prime \hookrightarrow C$ to the morphism $f_{C,C^\prime}$ respects composition follows similarly from the universal property of $f_{C,C^\prime}$. Our assignment is therefore a functor.
 
 It remains to prove that $f_{C,C^\prime}$ is a morphism of $R$-schemes. For $C^{\prime\prime}:=\{0\}$, we have $A_{C^{\prime\prime}}=\End(E_0) = R$ and hence $\cM(E_{C^{\prime\prime}})=\Spec R$. The result follows by commutativity of the morphisms $f_{C^{\prime},C^{\prime\prime}}\circ f_{C,C^\prime} = f_{C,C^{\prime\prime}}$ established above.
 \end{proof}
 
 Of particular interest to us are the following morphisms. 
 
 \begin{definition}
 \label{def:cornering}
  Let $C\subseteq \{0,1,\dots, n\}$ be a subset containing $0$. The \emph{cornering morphism} for $C$ is the universal morphism 
 \begin{equation} 
 \label{eqn:gC}
 g_C:= f_{\{0,\dots,n\},C}\colon \cM(E)\longrightarrow \cM(E_C) 
 \end{equation}
 obtained by applying the functor from Proposition~\ref{prop:functor} to the inclusion $C\hookrightarrow \{0,1,\dots,n\}$.
 \end{definition}

Theorem~\ref{thm:multigradedlinearseries} ensures that we understand the image of each cornering morphism. The next example illustrates that while these morphisms may be surjective, they need not be.

 \begin{example}
 For $\mathbb{P}^2$ equipped with the tilting bundle $E:=\mathcal{O} \oplus \mathcal{O}(1) \oplus \mathcal{O}(2)$, the algebra $\End_{\mathbb{P}^2}(E)$ can be presented using the Beilinson quiver with relations: 
  \begin{align*}
  \begin{aligned}
 \begin{tikzpicture}
\node (v0) at (0,0) {$0$};
\node (v0') at (0,0.3) {$\phantom{0}$};
\node (v0'') at (0,-0.3) {$\phantom{0}$};
\node (v1) at (3,0)  {$1$};
\node (v1') at (3,0.3) {$\phantom{1}$};
\node (v1'') at (3,-0.3) {$\phantom{1}$};
\node (v2) at (6,0)  {$2$};
\node (v2') at (6,0.3) {$\phantom{2}$};
\node (v2'') at (6,-0.3) {$\phantom{2}$};
\draw [->] (v0') to node[gap] {$\scriptstyle{x_0}$} (v1');
\draw [->] (v0) to node[gap] {$\scriptstyle{y_0}$} (v1);
\draw [->] (v0'') to node[gap] {$\scriptstyle{z_0}$} (v1'');
\draw [->] (v1') to node[gap] {$\scriptstyle{x_1}$} (v2');
\draw [->] (v1) to node[gap] {$\scriptstyle{y_1}$} (v2);
\draw [->] (v1'') to node[gap] {$\scriptstyle{z_1}$} (v2'');
\end{tikzpicture}
 \end{aligned}
 \qquad
 \begin{aligned}
x_1y_0&=y_1x_0 \\
y_1z_0&=z_1y_0 \\
z_1x_0&=x_1z_0
 \end{aligned}
 \end{align*}
 Consider the category $\mathscr{C}$ from Proposition~\ref{prop:functor} viewed as a poset: 
 \[
 \begin{tikzpicture}
\node (v0) at (0,0) {$ C_1:=\{0,1\} $};
\node (v1) at (2.5,1)  {$C_3=\{ 0, 1,2\}$};
\node (v2) at (2.5,-1)  {$C_0:= \{0\}$};
\node (v2) at (5,0)  {$C_2:= \{0, 2 \}$};

\node (v2) at (2.5,0)  {\rotatebox{90}{$\subset$}};
\node (v2) at (1.2,0.5)  {\rotatebox{45}{$\subset$}};
\node (v2) at (1.2,-0.5)  {\rotatebox{135}{$\subset$}};
\node (v2) at (3.7,0.5)  {\rotatebox{135}{$\subset$}};
\node (v2) at (3.7,-0.5)  {\rotatebox{45}{$\subset$}};
\end{tikzpicture}
 \]
 It is easy to calculate that $\mathcal{M}(E) \cong \mathbb{P}^2$. Example~\ref{exa:linearseries} implies that $g_{C_1}\colon\mathcal{M}(E) \rightarrow \mathcal{M}(E_{C_1})$ is an isomorphism, whereas $g_{C_2}\colon \mathcal{M}(E) \rightarrow \mathcal{M}(E_{C_2}) \cong \mathbb{P}^5$ is the Veronese embedding  $\varphi_{\vert\mathcal{O}(2)\vert}$. For $1\leq i\leq 3$, we have that $\mathcal{M}(E_{C_i}) \rightarrow \mathcal{M}(E_{C_0}) \cong \Spec \kk$ is the structure morphism.
 \end{example}
 
 In order to introduce the surjectivity criterion we continue to assume that $C\subseteq \{0,1,\dots, n\}$ contains $0$ and where the idempotent $e_C=\sum_{i\in C} e_i$ determines the algebra $A_C = e_C A e_C$. In this situation there are six functors forming a recollement of the abelian module category:
\[
\begin{tikzpicture}
\node (v0) at (0,0) {$\scriptsize{A/Ae_CA-\module}$};
\node (v1) at (4.5,0)  {$\scriptsize{A-\module}$};
\node (v2) at (9,0)  {$\scriptsize{A_C-\module}$};
\draw [->,bend right=15] (v1) to node[gap] {$\scriptstyle{i^*}$} (v0);
\draw [->,bend right=0] (v0) to node[gap] {$\scriptstyle{i_*}$} (v1);
\draw [->,bend left=15] (v1) to node[gap] {$\scriptstyle{i^{!}}$} (v0);

\draw [->,bend right=15] (v2) to node[gap] {$\scriptstyle{j_*}$} (v1);
\draw [->,bend right=0] (v1) to node[gap] {$\scriptstyle{j^*}$} (v2);
\draw [->,bend left=15] (v2) to node[gap] {$\scriptstyle{j_{!}}$} (v1);
\end{tikzpicture}
\]
where the functors are defined by 
\begin{align*}
\begin{aligned}
i^*(-)&:= A/(Ae_CA) \otimes_A (-) \\
i_*(-) &:=\inc \\
i^{!}(-) &:= \Hom_{A}(A/Ae_CA,-)
\end{aligned}\quad
\text{ and  }
\begin{aligned}\quad
j_*(-)&:=\Hom_{A_C}(eA,-) \\
j^*(-) &:=e_CA \otimes_A (-) \cong \Hom_{A}(Ae_C,-) \\
j_{!}(-) &:= Ae_C \otimes_{e_CAe_C} (-)
\end{aligned}
\end{align*}
such that $(i^*,i_*, i^{!})$ and $(j_{!},j^*,j_*)$ are adjoint triples, $i_*$, $j_*$, and $j_!$ are fully faithful, and $\im i_* = \ker j^*$. In particular, $j^*$ and $i_*$ are exact. The module $j_!(N)$ is maximally extended by objects supported on $A/AeA$. Indeed, for any $A/AeA$-module $M$, we have for all $k\in \ZZ$ that 
\begin{equation}
\label{eqn:Ext}
\Ext^k(j_!(N),i_*(M))=\Ext^k(N,j^*i_*(M))=0
\end{equation}

 Recall that in order to define the multigraded linear series $\cM(E_C)$ from \eqref{eqn:cornermorphism}, we introduced the dimension vector $\vv_C = (\rk(E_i))_{i\in C}$. 

\begin{lemma}
\label{lem:recollementProperties} 
 Let $N$ be an $A_C$-module of dimension vector $\vv_C$.
\begin{enumerate}
\item[\one] If $N$ is a $0$-generated $A_C$-module, then $j_!N$ is a $0$-generated $A$-module.
\item[\two] The $A$-module $j_! N$ is finite-dimensional and satisfies $\dim_i  j_!N = \dim_i N$ for $i \in C$. 
\end{enumerate}
\end{lemma}
\begin{proof}
 To simplify notation in this proof, write $e:=e_C$. For $i\in C$, let $e_i'$ denote the idempotent $e_i$ considered as an element of $A_C=eAe$. Since $N$ is $0$-generated there is a surjective map $A_Ce'_0 \rightarrow N$, and since $j_!$ is right exact there is a surjection $j_!(A_Ce'_0) \rightarrow j_!N$. We then calculate
\[
j_!(A_Ce'_0) := Ae \otimes_{eAe} A_C e^\prime_0 \cong Ae \otimes_{eAe} eAe \; e_0 \cong Ae e_0=Ae_0
\] 
as $0 \in C$, so there is a surjective map $Ae_0 \rightarrow j_!N$ which proves \one. For part \two, the algebra $A$ is a module-finite $e_0Ae_0$-algebra, so there are finitely many elements in $A$ that generate $A$ as an $e_0Ae_0$-algebra. In particular, there are finitely many elements $a_1, \dots, a_r$ of $Ae \subset A$ that generate $Ae$ as a right $eAe$-module. As a vector space, we have $Ae \otimes_{eAe} N \subset N^{\oplus r}$, so $N$ and hence also $j_!N=A e \otimes_{eAe}N$ is a finite dimensional vector space. Then for $i \in C$, we have $e_iAe \subset eAe$ and so $e_i j_! N = e_i Ae \otimes_{eAe} N = e'_iN$, giving $\dim_i j_! N = \dim_i N$ as required.
\end{proof}

Note that each closed point $x\in \mathcal{M}(E_C)$ determines a $0$-generated $A_C$-module $N_x$ of dimension vector $\vv_C$. 

\begin{proposition} 
\label{prop:submoduleSurjection} 
 A closed point $x \in \mathcal{M}(E_C)$ lies in the image of the cornering morphism $g_C\colon\mathcal{M}(E) \rightarrow \mathcal{M}(E_C)$ if and only if the $A$-module $j_!(N_x)$ admits a surjective map onto an $A$-module of dimension vector $\vv$. 
\end{proposition}
\begin{proof}
 Suppose that a closed point $x \in \mathcal{M}(E_C)$ lies in the image of $g_C\colon\mathcal{M}(E) \rightarrow \mathcal{M}(E_C)$. For any closed point $y \in \mathcal{M}(E)$ satisfying $g_C(y)=x$, the corresponding $0$-generated $A$-module $M_y$ satisfies $j^*M_y \cong N_x$. Let $C$ denote the cokernel of the counit map 
 \begin{equation}
 \label{eqn:counit}
 j_!j^*M_y \rightarrow M_y. 
\end{equation}
 If $C$ is nonzero, then the kernel of the surjection $M_y\rightarrow C$ is a proper $A$-submodule of $M_y$ that violates  $\theta$-stability of $M_y$ for any $\theta\in \Theta^+_{\vv}$. The counit map \eqref{eqn:counit} therefore gives the required surjective map to an $A$-module of dimension vector $\vv$.
 
  For the opposite direction, let $j_!(N_x)\to M$ be a surjective $A$-module homomorphism, where $M$ has dimension vector $\vv$. Since $N_x$ is $0$-generated, Lemma~\ref{lem:recollementProperties}\one\ gives that $j_!(N_x)$ is $0$-generated and hence so is $M$. As such, $M$ is a 0-generated $A$-module of dimension $\vv$, so $M$ is $M_y$ for some closed point $y \in \mathcal{M}(E)$.  Then $N_x \cong j^*j_!(N_x) \cong j^*(M_y)$, giving $g_C(y) =x$.
\end{proof}

The following result provides a homological criterion to check whether Proposition~\ref{prop:submoduleSurjection} applies.

\begin{lemma} 
\label{lem:nonzerohoms}
Suppose that $N$ is a 0-generated $A_C$-module of dimension vector $\vv_C$ and $M$ is a 0-generated $A$-module of dimension vector $\vv$. Then
\begin{enumerate}
\item[\one] if $\Hom_{A}(j_!(N),M) \neq 0$, then there exists a surjective map $j_!(N) \twoheadrightarrow M$; and 
\item[\two] if $\Hom_{A}(M,j_!(N)) \neq 0$, then there exists an isomorphism $M \cong j_!(N)$.
\end{enumerate}
\end{lemma}
\begin{proof}
A nonzero map $f\colon M \rightarrow M'$ between 0-generated $A$-modules with $\dim_0 M = \dim_0 M'=1$ is surjective, because a proper cokernel would contradict the 0-generated condition. Therefore if $\Hom_{A}(j_!(N),M) \neq 0$, Lemma~\ref{lem:recollementProperties}\one\ implies that there is a surjective map $j_!(N) \twoheadrightarrow M$.

Similarly if $\Hom_{A}(M,j_!(N)) \neq 0$ then there is a surjective map $M \twoheadrightarrow j_!(N)$ and hence a surjective map $j^*M \twoheadrightarrow j^*j_!(N) \cong N$. We have $\dim_i N= \dim_i M$ for $i \in C$ by assumption, so this surjective map is an isomorphism $j^*M\cong N$. Since $M$ is 0-generated, the counit map
\[
j_!(N) \cong j_!j^*M \rightarrow M
\]
is also surjective. The composition of the surjective maps $j_!(N) \twoheadrightarrow M \twoheadrightarrow j_!(N)$ gives the desired isomorphism $j_!(N) \cong M$.
\end{proof}

 \section{Cornering the reconstruction algebra}
 Let $G\subset \GL(2,\kk)$ be a finite subgroup that acts without pseudo-reflections. We now apply the results of the previous section to obtain a fine moduli space description of any partial resolution $Y^\prime$ of a quotient surface singularity $\mathbb{A}^2_\kk/G$ such that the minimal resolution $Y\to \mathbb{A}^2_\kk/G$ factors through $Y^\prime$. In fact we prove that every such $Y^\prime$ is the multigraded linear series for a summand of the tautological bundle on the $G$-Hilbert scheme. 

 Let $\Irr(G)$ denote the set of isomorphism classes of irreducible representations of $G$. The \emph{$G$-Hilbert scheme} is the fine moduli space of $G$-equivariant coherent sheaves of the form $\cO_Z$ for $Z\subset \mathbb{A}^2_\kk$ such that $\Gamma(\cO_Z)$ is isomorphic to the regular representation of $G$. The category of $G$-equivariant coherent sheaves is equivalent to the category of finitely generated modules over the skew group algebra $\kk[x,y]\rtimes G$, so the $G$-Hilbert scheme is isomorphic to the fine moduli space $\cM(\kk[x,y]\rtimes G,\vv,\theta)$, where $\vv=(\dim(\rho)^{\oplus \dim(\rho)})_{\rho\in \Irr(G)}$ and $\theta\in \Theta_{\vv}^+$ is a $0$-generated stability condition; here the trivial representation $\rho_0$ is the zero vertex.

Generalising the work of Ito--Nakamura~\cite{ItoNakamura99} for finite subgroups of $\SL(2,\kk)$, Kidoh~\cite{Kidoh01} and Ishii~\cite{Ishii02} proved that the $G$-Hilbert scheme is isomorphic to the minimal resolution $Y$ of $\mathbb{A}^2_\kk/G$.  The summands of the tautological bundle 
\[
T:=\bigoplus_{\rho\in \Irr(G)} T_\rho^{\oplus \dim(\rho)}
\]
 on the $G$-Hilbert scheme are globally generated by Corollary~\ref{cor:ggtautbundles}. To study the multigraded linear series of $T$, we need to know the endomorphism algebra of $T$.
 
 \begin{lemma} 
 \label{lem:TautologicalBundleEndomorphismAlgebra}
For a finite subgroup $G\subset\GL(2,\kk)$ without pseudo-reflections, the endomorphism algebra of $T$ is isomorphic to the skew group algebra $\kk[x,y] \rtimes G$. 
 \end{lemma}
 \begin{proof}
 For $\rho\in \Irr(G)$, Ishii~\cite[Corollary 3.2]{Ishii02} proves that the summand $T_\rho$ of $T$ is a globally generated full sheaf. The argument of Wemyss~\cite[Lemma 3.6]{Wemyss11} applies to globally generated full sheaves to give $\End_Y(T) \cong \End_{\kk[x,y]^G}(\pi_* T)$ for the resolution $\pi\colon Y \rightarrow \Spec \kk[x,y]^G$. By construction the $\kk[x,y]^G$-module $\pi_* T_{\rho}$ is isomorphic to $(\kk[x,y] \otimes \rho^{\vee})^G$.  Taking the sum over all $\rho^\vee\in \Irr(G)$ and relabelling, it follows from Auslander~\cite{AuslanderRationalSingularities} that 
 \[
 \End_Y(T)\cong \End_{\kk[x,y]^G} \left(\bigoplus_{\rho \in \text{Irr}(G)} \left( \left(\kk[x,y] \otimes \rho \right) ^G \right)^{\oplus \dim \rho} \right) \cong \kk[x,y] \rtimes G
 \]
 as required.
 \end{proof}

 It follows that the minimal resolution $Y$ of $\mathbb{A}^2_{\kk}/G$ is isomorphic to the multigraded linear series $\mathcal{M}(T)$. If we set $T_{\Irr}:=\bigoplus_{\rho\in \Irr(G)} T_\rho$ and define $\vv_{\Irr}=(\dim\rho)_{\rho\in \Irr(G)}$, then 
 \[
 B:=\End_Y(T_{\Irr})
 \]
 is Morita equivalent to the skew group algebra, and Lemma~\ref{lem:Morita} implies that $Y$ is isomorphic to $\cM(T_{\Irr})$; put another way, the universal morphism $f\colon Y \cong \cM(T) \to \cM(T_{\Irr})$ determined by the vector bundle $T_{\Irr}$ on $Y$ is an isomorphism. 
 
 Recall that an irreducible representation $\rho\in \Irr(G)$ is said to be \emph{special} if $H^1(T_{\rho}^{\vee})=0$. For the cornering set $\Spe(G):= \{ \rho \in \Irr(G) \mid \rho \text{ is special} \}$ and the idempotent $e:=\sum_{\rho\in \Spe(G)} e_\rho$, we obtain the \emph{reconstruction bundle}
 \[
 E:= T_{\Spe}=\bigoplus_{\rho\in \Spe(G)} T_\rho.
 \]
 Note that $T_{\Spe}\subseteq T_{\Irr}\subseteq T$ are subbundles. The argument in the proof of Lemma~\ref{lem:TautologicalBundleEndomorphismAlgebra} shows that the endomorphism algebra 
  \[
  A:= \End_Y(E)\cong eBe;
  \]
 of the reconstruction bundle $E$ is the \emph{reconstruction algebra} introduced by Wemyss~\cite{WemyssGL211} in general, and by Wemyss~\cite{Wemyss11} and Craw~\cite{CrawSpecial11} for a cyclic subgroup $G\subset \GL(2,\kk)$. The reconstruction bundle is of interest precisely because it is a tilting bundle by work of Van den Bergh~\cite{VdB04Duke}. Buchweitz--Hille~\cite[Proposition~2.6]{BuchweitzHille13} implies that the dual bundle $E^\vee$ is also a tilting bundle, so 
 \begin{equation}
 \label{eqn:ECdual}
 \Phi(-):=\mathbf{R}\!\Hom_Y(E^\vee,-)\colon D^b(Y) \longrightarrow D^b(A)
 \end{equation}
 is an equivalence of categories; put simply, the algebra $A$ enables one to reconstruct the derived category of the minimal resolution $Y$. We now show that $A$ enables us to reconstruct $Y$ itself.
  
\begin{proposition} 
\label{prop:reconstructionisomorphism}
 For a finite subgroup $G\subset\GL(2,\kk)$ without pseudo-reflections, the minimal resolution $Y$ of $\mathbb{A}^2_\kk/G$ is isomorphic to the multigraded linear series $\mathcal{M}(E)$, that is, to the fine moduli space of $0$-generated $A$-modules of dimension vector $\vv_{\Spe}=(\dim(\rho))_{\rho\in \Spe(G) }$.
\end{proposition}
\begin{proof}
 Since $Y\cong \mathcal{M}(T_{\Irr})$ and since $E$ is the summand of $T_{\Irr}$ corresponding to the cornering subset $\Spe(G)\subseteq \Irr(G)$, we need only prove that $g_{\Spe}\colon \mathcal{M}(T_{\Irr}) \rightarrow \mathcal{M}(T_{\Spe})$ is an isomorphism. The line bundle $\bigotimes_{\rho\in \Spe(G)} \det(T_\rho)$ has positive degree on each exceptional curve in $Y$, so it's ample. Applying Theorem~\ref{thm:multigradedlinearseries} and specifically Remark~\ref{rem:ample} shows that $g_{\Spe}$ is a closed immersion. It remains to prove that $g_{\Spe}$ is surjective. 
 
 Each point $x \in \mathcal{M}(T_{\Spe})$ corresponds to an $A$-module $N_x$ of dimension vector $\vv_{\Spe}$, and we now calculate the dimension vector of the $B$-module $j_!(N_x)$. The abelian category of finitely generated modules over $B$ has an indecomposable projective module $P_{\rho}$ and a one-dimensional vertex simple module $S_{\rho}$ for each $\rho\in \Irr(G)$. Since $\kk[x,y]\rtimes G$ is Morita equivalent to $B$, the observation of Ishii--Ueda~\cite[Proposition~1.1]{IshiiUeda15} gives a semiorthogonal decomposition
\[
D^b(B)\cong \big\langle \langle S_{\rho}\mid \rho \notin \Spe(G) \rangle, \langle P_{\rho} \mid \rho \in \Spe(G) \rangle\big\rangle,
\] 
 and the bundle $T_{\Irr}^\vee$ on the $G$-Hilbert scheme $Y$ defines a fully faithful embedding 
\[
\Phi_{\Irr}(-)=\mathbf{R}\!\Hom_Y(T_{\Irr}^\vee,-)\colon D^b(Y) \longrightarrow D^b(B) 
\]
with essential image $\langle P_{\rho} \mid \rho \in \Spe(G) \rangle$. Since $j_!(N_x)$ is left-orthogonal to $\langle S_\rho \mid \rho \notin \Spe(G) \rangle$ by \eqref{eqn:Ext}, it follows that $j_!(N_x)$ lies in the essential image of $\Phi_{\Irr}$ and $j_!(N_x) \cong \Phi_{\Irr}(\mathcal{E})$ for some $\mathcal{E} \in D^b(Y)$. The functor $\Phi_{\Irr}$ has a left adjoint $\Psi_{\Irr}$ such that $T_{\rho}^\vee\cong \Psi_{\Irr}(P_{\rho})$ and we calculate the dimension of $j_!(N_x)$ via the Euler form as
\begin{equation}
\label{eqn:dimj!Nx}
\dim_{\rho} j_!(N_x) = \chi_{B}(P_{\rho},j_!(N_x)) =  \chi_{B}(P_{\rho},\Phi_{\Irr}(\mathcal{E}))
= \chi_{Y}(\Psi_{\Irr}(P_{\rho}),\mathcal{E})
= \chi_{Y}(T_{\rho}^\vee,\mathcal{E})
\end{equation}
 for all $\rho\in \Irr(G)$. Lemma~\ref{lem:recollementProperties}\two\ gives $\dim_{\rho} j_!(N) = \dim \rho$ for $\rho \in \Spe(G)$, and hence
 \[
\chi_Y(T_{\rho}^{\vee},\mathcal{E}) = \dim \rho
 \]
for $\rho \in \Spe(G)$. Since the essential image of $\Phi_{\Irr}$ is generated by the indecomposable projectives $\{P_\rho \mid \rho \in \Spe(G)\}$, the classes of the bundles $T_\rho^\vee=\Psi_{\Irr}(P_\rho)$ for $\rho \in \Spe(G)$ generate the Grothendieck group $K(Y)$. By Lemma \ref{lem:KnumcY}, the Euler form $\chi_Y\colon K(Y)\times \Knumc(Y)\to \ZZ$ is a perfect pairing and hence as 
 \[
 \chi_Y(T_{\rho}^{\vee},\mathcal{E}) = \dim \rho = \rk T_{\rho}^{\vee}=  \chi_Y(T_{\rho}^{\vee},\mathcal{O}_y)
 \]
 for $\rho \in \Spe(G)$. It follows that
 \[
 [\mathcal{E}]=[\mathcal{O}_y] \in \Knumc(Y).
 \] 
As a consequence 
\[
\chi_{Y}(T_{\rho}^\vee,\mathcal{E})= \chi_{Y}(T_{\rho}^\vee,\mathcal{O}_y)=\rk T_\rho^\vee = \dim \rho
\]
 for all $\rho\in \Irr(G)$. Equation \eqref{eqn:dimj!Nx} now implies that $ \dim j_!(N_x)= (\dim\rho)_{\rho\in \Irr(G)}=\vv_{\Irr}$ for any $x\in \mathcal{M}(T_{\Spe})$. We deduce from Proposition \ref{prop:submoduleSurjection} that $g_{\Spe}\colon\mathcal{M}(T_{\Irr}) \rightarrow \mathcal{M}(T_{\Spe})$ is surjective.
\end{proof}

\begin{remark}
 Proposition~\ref{prop:reconstructionisomorphism} also follows from Karmazyn~\cite[Corollary~5.4.5]{Karmazyn14}, because the reconstruction bundle is a tilting bundle. When $G\subset \GL(2,\kk)$ is cyclic, an explicit construction of the isomorphism $Y\cong \mathcal{M}(A,\vv_{\Spe},\theta)$ for $\theta\in \Theta_{\vv_{\Spe}}^+$ was first given by \cite{Wemyss11,CrawSpecial11}.
\end{remark}

 We now strengthen Proposition~\ref{prop:reconstructionisomorphism} by providing a fine moduli construction for partial resolutions of $\mathbb{A}^2_\kk/G$ when $G\subset \GL(2,\kk)$ is a finite subgroup without pseudo-reflections. Let $E$ and $A=\End_Y(E)$ denote the reconstruction bundle and reconstruction algebra respectively, and let $C \subseteq \Spe(G)=\{\rho\in \Irr(G) \mid \rho \text{ is special}\}$ be any subset that contains the trivial representation $\rho_0$ of $G$. As in the proof of Proposition~\ref{prop:functor}, we obtain a morphism 
 \begin{equation}
 \label{eqn:reconstructioncontraction}
 g_{C} = f_{\Spe(G),C}\colon Y=\mathcal{M}(E)\longrightarrow \mathcal{M}(E_{C})
 \end{equation}
 from the minimal resolution $Y$ of $\mathbb{A}^2_{\kk}/G$. Again, Theorem~\ref{thm:multigradedlinearseries} shows that the image of $g_{C}$ is isomorphic to the image of the morphism to the linear series of $\bigotimes_{\rho\in C} \det(T_\rho)$. The construction of the tilting bundle $E$ on $Y$ \cite[Section~3.4]{VdB04Duke} implies that the invertible sheaves 
 \[
 \{\det(T_\rho)\mid \rho\in \Spe(G)\setminus \rho_0\}
 \]
 provide the integral basis of $\Pic(Y)$ dual to the curve classes defined by the irreducible components of the exceptional divisor of the resolution $Y\to \mathbb{A}^2_\kk/G$. It follows that the morphism $g_{C}$ from \eqref{eqn:reconstructioncontraction} contracts precisely those exceptional curves corresponding to irreducible representations $\rho\in \Spe(G)\setminus C$. 

\begin{theorem} 
\label{thm:surjectioncontraction}
 Let $G\subset \GL(2,\kk)$, be a finite subgroup without pseudo-reflections, and let $Y^\prime$ be any partial resolution of $\mathbb{A}^2_\kk/G$ such that the minimal resolution $Y\to \mathbb{A}^2_\kk/G$ factors via $Y^\prime$. There is a cornering set $C \subseteq \Spe(G)$ containing $\rho_0$ such that $Y^\prime$ is isomorphic to $\mathcal{M}(E_{C})$.  
\end{theorem}
\begin{proof}
 The partial resolution $Y^\prime$ is obtained by contracting a set of components of the exceptional divisor in $Y$. Remove from $\Spe(G)$ those $\rho$ such that $\det(T_\rho)$ is dual to one of the curves being contracted, leaving a subset $C\subset \Spe(G)$ containing $\rho_0$ such that the contraction $Y\to Y^\prime$ contracts precisely the same curves as the morphism $g_{C}$ from \eqref{eqn:reconstructioncontraction}. The result follows once we prove that $g_{C}$ is surjective. 
 
 By Proposition \ref{prop:submoduleSurjection}, we need only show that for any $0$-generated $A_{C}$-module $N$ of dimension $\vv_{C}:=(\dim \rho)_{\rho\in C}$, there is a surjective $A$-module homomorphism $j_!(N) \twoheadrightarrow M$, where $M$ has dimension $\vv_{\Spe}$. For any such $N$, the derived equivalence $\Phi$ from \eqref{eqn:ECdual} gives $\mathcal{E} \in D^b(Y)$ such that $\Phi(\mathcal{E}) \cong j_!(N)$, and \cite[Proposition~7.2.1]{BayerCrawZhang16} implies that $\mathcal{E}$ has proper support because $j_!(N)$ is finite dimensional over $\kk$. Suppose that $j_!(N)$ does not admit a surjective $A$-module homomorphism $j_!(N) \twoheadrightarrow M$, where $M$ is $0$-generated of dimension $\vv_{\Spe}$. Lemma \ref{lem:nonzerohoms} implies that
\begin{equation}
\label{eqn:HomAC}
0=\Hom_{A}(j_!(N),M)=\Hom_{A}(M,j_!(N)).
\end{equation}
Every $0$-generated $A$-module $M$ of dimensional vector $\vv_{\Spe}$ is of the form $\Phi(\mathcal{O}_y)$ for some $y \in Y$. Applying the quasi-inverse of $\Phi$ to \eqref{eqn:HomAC} gives
\[
0=\Hom_{D^b(Y)}(\mathcal{E},\mathcal{O}_y)=\Hom_{D^b(Y)}(\mathcal{O}_y,\mathcal{E})
\]
for all closed points $y \in Y\cong\mathcal{M}(E)$. Since $Y$ is smooth, there is a dualising line bundle $\omega_Y$ that induces Serre duality on objects with compact support such that 
\[
0 =\Hom_{D^b(Y)}(\mathcal{O}_y,\mathcal{E})  =\Hom^2_{D^b(Y)}(\mathcal{E}, \omega_Y \otimes\mathcal{O}_y)=\Hom^2_{D^b(Y)}(\mathcal{E}, \mathcal{O}_y)
\]
 for all closed points $y \in Y$. Bridgeland--Maciocia~\cite[Proposition 5.4]{BridgelandMaciocia02} implies that the object $\mathcal{E}$ has homological dimension $0$. The codimension of the support of $\mathcal{E}$ is bounded above by its homological dimension \cite[Corollary 5.5]{BridgelandMaciocia02}, so the support of $\mathcal{E}$ must be of codimension $0$; this is a contradiction, because $\mathcal{E}$ has proper support. This completes the proof.
\end{proof}

\begin{remark}
 An analogous result in the complete local setting can be deduced by combining Karmazyn~\cite[Corollary~5.2.5]{Karmazyn14} with Iyama--Kalck--Wemyss--Yang~\cite[Theorem~4.6]{KIWY}. 
 \end{remark}
 
\begin{example} 
\label{ex:McKayReconstruction}
To illustrate Proposition~\ref{prop:reconstructionisomorphism} and Theorem~\ref{thm:surjectioncontraction}, consider the subgroup 
\[
G:=\left\langle \left( \begin{array}{c c} \varepsilon_4 & 0 \\ 0 & \varepsilon_4^{-1} \end{array} \right) , \left( \begin{array}{c c} 0 &\varepsilon_4 \\ \varepsilon_4 & 0 \end{array} \right), \left( \begin{array}{c c} 0 &\varepsilon_6 \\ \varepsilon_6 & 0 \end{array} \right) \right \rangle \subset \GL(2,\kk),
\]
where $\varepsilon_r$ is a primitive $r^{th}$ root of unity. This group $G$ is the direct product of the quaternion group of order 8 with a cyclic group of order 3; it has 24 elements and 15 irreducible representations, 3 of dimension two and 12 of dimension one. Below we draw: (a) the McKay quiver (taken with the mesh relations) where we mark the special representations labelled $0,1,\dots, 4$; and (b) the Special McKay quiver with the given relations which provides a presentation for the reconstruction algebra $A$. 
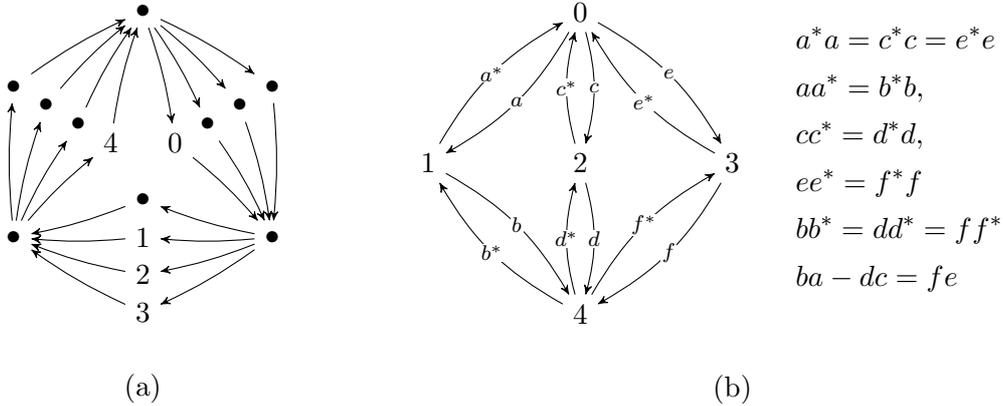
\begin{align*}
\begin{aligned}
\begin{tikzpicture} [bend angle=5, looseness=1]
\node (C1) at (0,2)  {$\bullet$};
\node (C2) at (1.7,-1)  {$\bullet$};
\node (C3) at (-1.7,-1)  {$\bullet$};
\node(B1) at (0.425,0.25) {$0$};
\node(B2) at (0.85,0.5) {$\bullet$};
\node(B3) at (1.275,0.75) {$\bullet$};
\node(B4) at (1.7,1) {$\bullet$};
\node(D1) at (0,-0.5) {$\bullet$};
\node(D2) at (0,-1) {$1$};
\node(D3) at (0,-1.5) {$2$};
\node(D4) at (0,-2) {$3$};
\node(F) at (0,-3){(a)};
\node(E1) at (-0.425,0.25) {$4$};
\node(E2) at (-0.85,0.5) {$\bullet$};
\node(E3) at (-1.275,0.75) {$\bullet$};
\node(E4) at (-1.7,1) {$\bullet$};
\draw [->,bend left] (C1) to node[]  {} (B1);
\draw [->,bend left] (C1) to node[]  {} (B2);
\draw [->,bend left] (C1) to node[]  {} (B3);
\draw [->,bend left] (C1) to node[]  {} (B4);
\draw [->,bend left] (B1) to node[]  {} (C2);
\draw [->,bend left] (B2) to node[]  {} (C2);
\draw [->,bend left] (B3) to node[]  {} (C2);
\draw [->,bend left] (B4) to node[]  {} (C2);
\draw [->,bend left] (C2) to node[]  {} (D1);
\draw [->,bend left] (C2) to node[]  {} (D2);
\draw [->,bend left] (C2) to node[]  {} (D3);
\draw [->,bend left] (C2) to node[]  {} (D4);
\draw [->,bend left] (D1) to node[]  {} (C3);
\draw [->,bend left] (D2) to node[]  {} (C3);
\draw [->,bend left] (D3) to node[]  {} (C3);
\draw [->,bend left] (D4) to node[]  {} (C3);
\draw [->,bend left] (C3) to node[]  {} (E1);
\draw [->,bend left] (C3) to node[]  {} (E2);
\draw [->,bend left] (C3) to node[]  {} (E3);
\draw [->,bend left] (C3) to node[]  {} (E4);
\draw [->,bend left] (E1) to node[]  {} (C1);
\draw [->,bend left] (E2) to node[]  {} (C1);
\draw [->,bend left] (E3) to node[]  {} (C1);
\draw [->,bend left] (E4) to node[]  {} (C1);
\end{tikzpicture}
\end{aligned}
& \quad \quad \quad \quad
\begin{aligned}
\begin{tikzpicture} [bend angle=15, looseness=1]
\node (C1) at (0,0)  {$0$};
\node (C2) at (-2,-2)  {$1$};
\node (C3) at (0, -2){$2$};
\node (C4) at (2,-2) {$3$};
\node (C5) at (0,-4) {$4$};
\node (D) at (2,-5) {(b)};
\draw [->,bend left] (C1) to node[gap]  {\scriptsize{$a$}} (C2);
\draw [->,bend left] (C2) to node[gap]  {\scriptsize{$a^*$}} (C1);
\draw [->,bend left] (C2) to node[gap]  {\scriptsize{$b$}} (C5);
\draw [->,bend left] (C5) to node[gap]  {\scriptsize{$b^*$}} (C2);
\draw [->,bend left] (C1) to node[gap]  {\scriptsize{$c$}} (C3);
\draw [->,bend left] (C3) to node[gap]  {\scriptsize{$c^*$}} (C1);
\draw [->,bend left] (C3) to node[gap]  {\scriptsize{$d$}} (C5);
\draw [->,bend left] (C5) to node[gap]  {\scriptsize{$d^*$}} (C3);
\draw [->,bend left] (C1) to node[gap]  {\scriptsize{$e$}} (C4);
\draw [->,bend left] (C4) to node[gap]  {\scriptsize{$e^*$}} (C1);
\draw [->,bend left] (C4) to node[gap]  {\scriptsize{$f$}} (C5);
\draw [->,bend left] (C5) to node[gap]  {\scriptsize{$f^*$}} (C4);
\end{tikzpicture}
\end{aligned}
\quad
\begin{aligned}
&a^*a=c^*c=e^*e\ \\
&aa^*=b^*b, \\
&cc^*=d^*d, \\
&ee^*=f^*f \\
&bb^*=dd^*=ff^*\\
&ba-dc=fe\\
& \\
&
\end{aligned}
\end{align*}
(The mesh relations kill the paths between vertices 0 and 4 that don't factor through 1, 2, or 3, so there are no arrows in the Special McKay quiver between 0 and 4.)  In this case, the exceptional divisor of the minimal resolution $Y\to\mathbb{A}^2_\kk/G$ consists of three $(-2)$-curves meeting a central $(-3)$-curve as shown. 
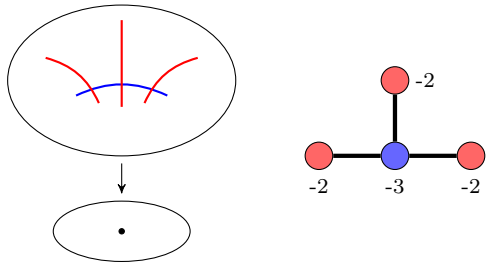
\begin{align*}
\begin{aligned}
\begin{tikzpicture}   
\draw[blue,thick] (-0.6,-0.2) to [bend left=25] node[pos=0.48, right] {} (0.6,-0.2);
\draw[red,thick] (0,-0.35) to [bend right=0] node[pos=0.48, right] {} (0,0.8);
\draw[red, thick] (-1,0.3) to [bend left=25] node[pos=0.48, right] {} (-0.3,-0.3);
\draw[red,thick] (1,0.3) to [bend right=25] node[pos=0.48, right] {} (0.3,-0.3);
\draw[black] (0,0) ellipse (1.5 and 1);
\draw[->] (0,-1.1) -- (0,-1.5);
\draw[black] (0,-2) ellipse (0.9 and 0.4);
\filldraw [black] (0,-2) circle (1pt);
\end{tikzpicture} 
\end{aligned}
&
\qquad
\begin{aligned}
\begin{tikzpicture}[node distance=1cm, main node/.style={circle,fill=red!60,draw,font=\sffamily\Large\bfseries}]
  \node[circle,fill=blue!60,draw] (1) at (0,0) {};
  \node[main node] (2) at (-1,0) {};
  \node[main node] (3) at (1,0) {};
  \node[main node] (4)  at (0,1) {};
\node (1l) at (0,-0.4) {\scriptsize{-3}};
\node (2l) at (-1,-0.4) {\scriptsize{-2}};
\node (3l) at (1,-0.4) {\scriptsize{-2}};
\node (4l) at (0.4,1) {\scriptsize{-2}};
 \draw [ultra thick] (1) to node {} (2);
 \draw [ultra thick] (1) to node {} (3);
 \draw [ultra thick] (1) to node {} (4);
\end{tikzpicture}
\end{aligned}
\end{align*}
 The minimal resolution $Y$ can be constructed either as the $G$-Hilbert scheme, whose tautological bundle $T$ has 15 non-isomorphic summands, or as the fine moduli space $\mathcal{M}(E)$ of $0$-generated modules over the reconstruction algebra, in which case the tautological bundle $\bigoplus_{\rho\in \Spe(G)} T_\rho$ has 5 non-isomorphic summands. 

To illustrate Theorem~\ref{thm:surjectioncontraction}, consider the cornering subset $C=\{0,1,2,3\}$ of the set $\Spe(G)=\{0,1,2,3,4\}$ of special representations of $G$. The algebra $A_{C}$ obtained from the reconstruction algebra $A$ can be presented using the following quiver with relations: 
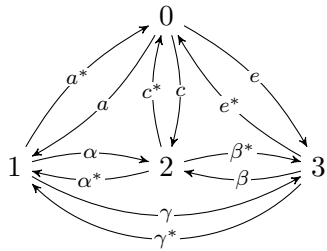
\begin{align*}
\begin{aligned}
\begin{tikzpicture} [bend angle=15, looseness=1]
\node (C1) at (0,0)  {$0$};
\node (C2) at (-2,-2)  {$1$};
\node (C3) at (0, -2){$2$};
\node (C4) at (2,-2) {$3$};
\draw [->,bend left] (C1) to node[gap]  {\scriptsize{$a$}} (C2);
\draw [->,bend left] (C2) to node[gap]  {\scriptsize{$a^*$}} (C1);
\draw [->,bend left] (C1) to node[gap]  {\scriptsize{$c$}} (C3);
\draw [->,bend left] (C3) to node[gap]  {\scriptsize{$c^*$}} (C1);
\draw [->,bend left] (C3) to node[gap]  {\scriptsize{$\alpha^*$}} (C2);
\draw [->,bend left] (C2) to node[gap]  {\scriptsize{$\alpha$}} (C3);
\draw [->,bend left] (C1) to node[gap]  {\scriptsize{$e$}} (C4);
\draw [->,bend left] (C4) to node[gap]  {\scriptsize{$e^*$}} (C1);
\draw [->,bend left] (C4) to node[gap]  {\scriptsize{$\beta$}} (C3);
\draw [->,bend left=45] (C4) to node[gap]  {\scriptsize{$\gamma^*$}} (C2);
\draw [->,bend left] (C3) to node[gap]  {\scriptsize{$\beta^*$}} (C4);
\draw [->,bend right=30] (C2) to node[gap]  {\scriptsize{$\gamma$}} (C4);
\end{tikzpicture}
\end{aligned}
\quad
\begin{aligned}
&a^*a=c^*c=e^*e \\
&\alpha^* \alpha =\gamma^* \gamma = a a^* a a^*, \\
&\alpha \alpha^* = \beta \beta^*=cc^*cc^*, \\
&\gamma \gamma^*=\beta^*\beta = ee^*ee^* \\
& aa^*a=\alpha^*c+\gamma^* e, \\
& cc^*c = \alpha a - \beta e, \\
& ee^*e = \gamma a - \beta^* c.
\end{aligned}
&
\qquad
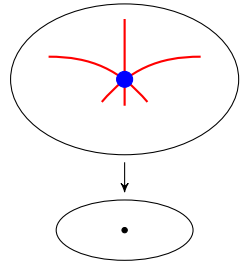
\begin{aligned}
\begin{tikzpicture}   
\draw[red,thick] (0,-0.35) to [bend right=0] node[pos=0.48, right] {} (0,0.8);
\draw[red, thick] (-1,0.3) to [bend left=25] node[pos=0.48, right] {} (0.3,-0.3);
\draw[red,thick] (1,0.3) to [bend right=25] node[pos=0.48, right] {} (-0.3,-0.3);
\draw[black] (0,0) ellipse (1.5 and 1);
\draw[->] (0,-1.1) -- (0,-1.5);
\draw[black] (0,-2) ellipse (0.9 and 0.4);
\filldraw [black] (0,-2) circle (1pt);
\filldraw [blue] (0,0) circle (3pt);
\end{tikzpicture} 
\end{aligned}
\end{align*}
The morphism $g_{C}\colon Y=\mathcal{M}(E) \rightarrow Y^\prime=\mathcal{M}(E_{C})$ is surjective, and it contracts the central (-3)-curve of the exceptional divisor in the minimal resolution. \end{example}

 \section{Cornering noncommutative crepant resolutions} \label{Section:NCCR}
 In this section we recall Van den Bergh's notion of an NCCR~\cite{VdB04} and show that the moduli spaces determined by $0$-generated stability parameters are multigraded linear series. We establish a set of sufficient conditions for the cornering morphisms to be surjective, and we demonstrate that these conditions hold in a range of situations. Throughout this section, we assume that $R$ is a normal, Gorenstein $\kk$-algebra of Krull dimension at most three. 

 Recall that a $\kk$-algebra $A$ is a \emph{noncommutative crepant resolution} (NCCR) \emph{of $R$} if there exists a reflexive $R$-module $M$ such that $A:=\End_R(M)$ is Cohen--Macaulay as an $R$-module and is of finite global dimension. Choose a decomposition $M \cong \bigoplus_i M_i$ into finitely many indecomposable, reflexive $R$-modules;  in general this decomposition is non-unique. Given one such decomposition $A= \End_R(\bigoplus_{i}M_i)$, we obtain a presentation as a quotient $A\cong \kk Q/I$ exactly as for the geometric setting described following equation \eqref{eqn:quiveralgebra}. We impose the following additional standing assumption on $A=\End_R(M)$:
 
\begin{assumption}
 \label{ass:KrullSchmidt}
 \begin{enumerate}
 \item[\one] All indecomposable projective $A$-modules occur as summands of $A$, and the presentation $A=\kk Q/I$ corresponds to a unique decomposition of $M= \bigoplus_{i \in Q_0} M_i$ into non-isomorphic, indecomposable reflexive modules; and
 \item[\two] the ideal $I\subset \kk Q$ is generated by linear combinations of paths of length at least one.
 \end{enumerate}
\end{assumption}
 
 \begin{remarks}
  \label{rems:idealidempotent}
 \begin{enumerate}
 \item If the module category of $R$ has the Krull-Schmidt property, then any NCCR $A'$ is Morita equivalent to an NCCR $A$ satisfying Assumption \ref{ass:KrullSchmidt}\one.
\item Assumption~\ref{ass:KrullSchmidt}\two\ ensures that each vertex $i\in Q_0$ determines a vertex simple $A$-module $S_i=\kk e_i$, and that distinct summands of $M$ are non-isomorphic (as $M_i\cong M_j$ would force the relations $aa^\prime-e_i, a^\prime a-e_j\in I$ for some $a, a^\prime \in Q_1$).
\end{enumerate}
\end{remarks}
    
  Define the dimension vector $\vv = (v_i)\in \ZZ^{Q_0}$ by setting $v_i=\rk_R(M_i)$ for $0\leq i\leq n$. After replacing $A$ by a Morita equivalent algebra if necessary, Van den Bergh~\cite[Section 6.3]{VdB04} notes that we may assume that $\vv$ is indivisible. For any generic stability parameter $\theta\in \Theta_{\vv}$, the tautological bundle $T=\bigoplus_{0\leq i\leq n}T_i$ on the moduli space $\cM(A,\vv,\theta)$ is a left $A$-module, so $T^\vee$ is a right $A$-module. When $A$ satisfies Assumption~\ref{ass:KrullSchmidt}, \cite[Remark~6.6]{VdB04} observes that the approach of Bridgeland--King--Reid~\cite{BKR01} implies that $\cM(A,\vv,\theta)$ is connected; we provide a slightly simplified proof of this observation in Proposition~\ref{prop:connected}.  
  
   Applying \cite[Theorem~6.3.1, Remark~6.6.1]{VdB04} gives a morphism 
 \begin{equation}
 \label{eqn:crepresolution}
 f\colon Y:= \mathcal{M}(A,\vv,\theta)\longrightarrow \Spec R 
 \end{equation}
 that is a projective crepant resolution. Again, \cite[Proposition~2.6]{BuchweitzHille13} implies that the bundle $T^\vee=\bigoplus_{0\leq i\leq n} T^\vee_i$ dual to the tautological bundle gives a derived equivalence 
 \begin{equation}
 \label{eqn:derivedequivPhi}
 \Phi(-):= \mathbf{R}\!\Hom_Y(T^\vee,-)\colon D^b(Y)\longrightarrow D^b(A)
 \end{equation}
 with quasi-inverse
 \begin{equation}
 \label{eqn:derivedequivalenceadjoint}
 \Psi(-):= T^\vee\otimes_{A}- \colon D^b(A)\longrightarrow D^b(Y).
 \end{equation}
 To clarify our conventions, $\Phi(T^\vee) = \End(T^\vee)$ is a right $\End(T^\vee)$-module and hence a left module over $\End(T^\vee)^{\text{op}}\cong \End(T) \cong A$. Under the equivalences \eqref{eqn:derivedequivPhi} and \eqref{eqn:derivedequivalenceadjoint}, the full subcategory $D_c^b(Y)$ in $D^b(Y)$ of objects with proper support is equivalent to the full subcategory $\Dfinb(A)$ in $D^b(A)$ of finite-dimensional left $A$-modules (see \cite[Lemma~7.1.1]{BayerCrawZhang16}).
 
 In order to apply the ideas of Sections \ref{Section:MultigradedLinearSeries}--\ref{Section:Thecorneringcategoryandrecollement} to an NCCR $A:=\End_R(M) \cong \kk Q/I$ satisfying Assumption \ref{ass:KrullSchmidt}, we suppose in addition that $M_0=R$. 
 Since $A\cong \End(T)$, the multigraded linear series of the tautological bundle $T$ is $\mathcal{M}(T):= \mathcal{M}(A,\vv,\theta)$ for any $0$-generated stability condition $\theta \in \Theta_{\vv}^{+}$, and hence  
 \begin{equation}
 \label{eqn:fcrepres}
 f\colon Y=\mathcal{M}(T)\longrightarrow \Spec R 
\end{equation}
 is a projective, crepant resolution. 
 
 \begin{remark}
 This construction picks out one crepant resolution of $\Spec R$. However:
 \begin{enumerate}
 \item if more than one summand $M_i$ of $M$ has rank one, then by relabelling $i$ to be the zero vertex and repeating the construction, the tautological bundle $T^\prime$ on the resulting moduli space determines a projective crepant resolution $f^\prime\colon \mathcal{M}(T^\prime)\to \Spec R$ that need not be isomorphic to that from \eqref{eqn:fcrepres}, e.g. see the suspended pinch point example considered in \cite[Example B.8.7]{Bocklandt16} where the choice of zero vertex as the vertex currently labelled 1 or 2 yields two crepant resolutions that are not isomorphic as varieties;
 \item for a finite subgroup $G\subset \SO(3)$, the skew group algebra $A=\kk[x,y,z]\rtimes G$ is an NCCR of $R:=\kk[x,y,z]^G$. Nolla de Celis--Sekiya~\cite{NollaSekiya11} show that every projective crepant resolution of $\Spec R$ is of the form $\mathcal{M}(A^\prime, \vv^\prime, \theta^\prime)$, where $A^\prime$ is an algebra obtained from $A$ by a sequence of mutations at vertices (none of which is the zero vertex), where $\vv$ is a dimension vector and where $\theta^\prime$ is a $0$-generated stability condition. In particular, since the endomorphism algebra of the tautological bundle on $\mathcal{M}(A^\prime, \vv^\prime, \theta^\prime)$ is isomorphic to $A^\prime$, the multigraded linear series of the tautological bundle is isomorphic to $\mathcal{M}(A^\prime, \vv^\prime, \theta^\prime)$ and every projective crepant resolution of $\Spec R$ can be constructed as such a multigraded linear series; and
 \item if $R$ is a complete local Gorenstein algebra over the field $\mathbb{C}$ such that $\Spec R$ admits a projective crepant resolution whose fibres have dimension at most one, then as in (2) above, work of Wemyss~\cite{Wemyss14} implies that every projective crepant resolution of $\Spec R$ is isomorphic to a multigraded linear series of a tautological bundle.
 \end{enumerate}
 \end{remark}
 
A choice of cornering subset $\{0 \} \subset C \subset Q_0$ produces the corresponding cornering morphism 
\[
g_C\colon Y:=\mathcal{M}(T) \longrightarrow Y':=\mathcal{M}(T_C)
\]
 whose image is isomorphic to that of the morphism $\varphi_{\vert L\vert} \colon Y\to \vert L \vert$ for $L:= \bigotimes_{i\in C} \det(T_i)$. Each closed point $x \in Y'$  determines a $0$-generated $A_C$-module $N_x$ of dimension $\vv_C$. To state the next result, we recall in the appendix that $\Knumc(Y)$ is the numerical Grothendieck for compact support introduced in \cite{BayerCrawZhang16}.

\begin{theorem}
\label{thm:recollement2}
 Let $A=\End(\bigoplus_{i\in Q_0} M_i)$ be an \emph{NCCR} satisfying Assumption~\ref{ass:KrullSchmidt} with $M_0\cong R$. Suppose that for all $x \in Y'$,  there exist $a_i \in \mathbb{Z}_{>0}$ and sheaves $F_i$ with proper support such that 
 \[
[\Psi(j_!(N_x))]= \sum_{i=0}^m a_i [F_i]\in \Knumc(Y).
 \]
 Then the morphism $g_C\colon Y \rightarrow Y'$ is surjective.
\end{theorem}
\begin{proof}
To simplify notation, write $\mathcal{E}_x:=\Psi(j_!(N_x))$. If we suppose that the result is false, then Proposition~\ref{prop:submoduleSurjection} and Lemma~\ref{lem:nonzerohoms} give $x\in Y^\prime$ such that 
\begin{equation}
\label{eqn:HomNCCR}
0=\Hom_{A}(j_!(N_x),M)=\Hom_{A}(M,j_!(N_x))
\end{equation}
for every $0$-generated $A$-module $M$ of dimension vector $\vv$. Every such module is of the form $M=\Phi(\mathcal{O}_y)$ for some closed point $y \in Y \cong \mathcal{M}(T)$. Apply $\Psi(-)$ to \eqref{eqn:HomNCCR} to obtain 
\[
0=\Hom_{D^b(Y)}(\mathcal{E}_x,\mathcal{O}_y)=\Hom_{D^b(Y)}(\mathcal{O}_y,\mathcal{E}_x)
\]
 for every closed point $y \in Y$.  Serre duality gives $\Hom^0_{D^b(Y)}(\mathcal{E}_x,\mathcal{O}_y)=\Hom^3_{D^b(Y)}(\mathcal{E}_x,\mathcal{O}_y)=0$, so $\Hom^j_{D^b(Y)}(\mathcal{E}_x,\mathcal{O}_y)=0$ unless $1 \le j \le 2$. Bridgeland--Maciocia~\cite[Proposition~5.4]{BridgelandMaciocia02} implies that $\mathcal{E}_x$ has homological dimension at most one, so $\mathcal{E}_x$ is quasi-isomorphic to a complex 
 \[
 L^{-2}\stackrel{g}{\longrightarrow} L^{-1},
 \]
 where $L^{-i}$ is a vector bundle in degree $-i$. Since $\mathcal{E}_x\in D_c^b(Y)$, we have that $H^{-1}(\mathcal{E}_x)\cong \ker(g)$ is a torsion subsheaf of $L^{-2}$ and hence is zero. It follows that $\mathcal{E}_x \cong F[1]$ where $F$ is the cokernel of $g$. In particular, $\mathcal{E}_x$ is quasi-isomorphic to the shift of a sheaf and $[\mathcal{E}_x]=-[F] \in \Knumc(Y)$. 
 
 By assumption, the class $[\mathcal{E}_x] \in \Knumc (Y)$ equals a positive sum of classes of sheaves, so
\[
\sum a_i [F_i] =[\mathcal{E}_x] = -[F]\in \Knumc(Y),
\]
where each $F_i$ is a sheaf on $Y$ and $a_i \in \mathbb{Z}_{>0}$. Let $j\colon Y\to\overline{Y}$ be a smooth projective completion. By \cite[Proof of Lemma 5.1.1]{BayerCrawZhang16}, we may pushforward numerical classes with compact support to obtain $\sum a_i[j_*F_i] = -[j_*F] \in \Knum(\overline{Y})$. However, for any sufficiently ample bundle $L$ on $\overline{Y}$, the integers $\sum a_i\chi(j_*F_i\otimes L)$ and  $\chi(j_*F\otimes L)$ are both positive, a contradiction.
\end{proof}

\begin{remark}
\label{rem:connectedness}
 This proof is adapted from that of Proposition~\ref{prop:connected} which in turn is adapted from the connectedness result of Bridgeland--King--Reid~\cite[Section~8]{BKR01}; our use of the numerical Grothendieck group enables us to bypass \cite[Lemma~8.1]{BKR01}.
 \end{remark}

This is particularly applicable when the vertex simple $A$-modules $S_i$ for $i \in Q_0 \backslash C$ are well understood and when the dimension of $j_!(N_x)$ can be calculated.

\begin{corollary} 
\label{cor:surjective}
  Let $A=\End(\bigoplus_{i\in Q_0} M_i)$ be an \emph{NCCR} satisfying Assumption~\ref{ass:KrullSchmidt} with $M_0\cong R$. Suppose in addition that either:
\begin{enumerate}
\item[\one] $\dim_i j_!(N_x) = v_i$ for all $i \in Q_0\setminus C$ and all $x \in Y'$; or 
\item[\two] $\Psi(S_i)$ is a sheaf for all $i \in Q_0\setminus C$ and $\dim_i j_!(N_x)  \ge v_i$ for all $i\in Q_0\setminus C$ and all $x \in Y'$.
\end{enumerate}
Then Theorem~\ref{thm:recollement2} applies, so $g_C\colon Y \rightarrow Y'$ is surjective.
\end{corollary}
\begin{proof}
 By Lemma~\ref{lem:KnumANCCR}, the numerical class of a finite dimensional $A$-module $M$ is determined by its dimension vector $[M] =\sum_{i\in Q_0} \dim M_i [S_i]$. The skyscraper sheaf for each closed point $y\in Y$ satisfies $\mathcal{O}_y=\Psi(M_y)$ for some $0$-generated $A$-module $M_y$ of dimension vector $\vv$, so
 \[
 [\mathcal{O}_y] = \sum_{i\in Q_0} v_i [\Psi(S_i)].
 \]
 To apply Theorem~\ref{thm:recollement2}, we express each class $[\mathcal{E}_x]$ as a positive sum of classes of sheaves with proper support. Indeed, in case \one, Lemma~\ref{lem:recollementProperties}\two\ gives $\dim_i j_!(N_x) = v_i$ for all $i \in C$ and all $x \in Y'$, so the assumption in case \one\ gives $[j_!(N_x)] = \vv \in \Knum(A)$ and hence 
 \[
 [\mathcal{E}_x]= [\Psi(j_!(N_x))] = [\Psi(\vv)] = [\mathcal{O}_y]
 \]
 for each point $y\in Y$. Similarly, in case \two, write $w_i:= \dim_i j_!(N_x)-v_i$ for $i\in Q_0$. Note that $w_i=0$ for $i\in C$ by Lemma~\ref{lem:recollementProperties}\two, and $w_i\geq 0$ for $i\in Q_0\setminus C$ by assumption. Then for all $x\in Y^\prime$, we have that $[j_!(N_x)] = \sum_{i\in Q_0} (v_i+w_i)[\Psi(S_i)]$, so 
 \[
 [\mathcal{E}_x]= [\Psi(j_!(N_x))] =  \sum_{i\in Q_0} (v_i+w_i) [\Psi(S_i)] = [\mathcal{O}_y] +  \sum_{i\in Q_0\setminus C} w_i [\Psi(S_i)]
 \]
 for each point $y\in Y$. In either case, the result is immediate from Theorem~\ref{thm:recollement2}.
  \end{proof}

\begin{remark}
In fact, the proof of Corollary~\ref{cor:surjective}\two\ shows that one requires only that the class $[\Psi(S_i)]\in \Knumc(Y)$ is a non-negative combination of classes of sheaves for each $i\in Q_0\setminus C$. 
\end{remark}

\section{The toric case in dimension three}
 We now specialise to the case where $R$ is a Gorenstein, semigroup algebra of dimension three, so $X=\Spec R$ is a Gorenstein toric threefold. Ishii--Ueda~\cite{IshiiUeda09} and Broomhead~\cite{Broomhead12} show that $R$ admits a noncommutative crepant resolution $A=\End_R(\bigoplus_{i\in Q_0} M_i)$ obtained as the Jacobian algebra $\kk Q/I$ of a quiver with potential arising from a consistent dimer model on a real two-torus. Toric algebras of this form necessarily satisfy Assumption~\ref{ass:KrullSchmidt}; in fact, the conclusions of Lemma~\ref{lem:KnumANCCR} were noted first by Ishii--Ueda~\cite[Proposition~8.3]{IshiiUeda13} in this context.
 
 For the dimension vector $\vv\in \Knum(A)$ with $v_i=1$ for all $i\in Q_0$, choose once and for all a vertex $0\in Q_0$ and a stability parameter $\theta\in \Theta_{\vv}^+$. If we write $T=\bigoplus_{i\in Q_0} T_i$ for the tautological bundle on the fine moduli space $\mathcal{M}(A,\vv,\theta)$, then as in equations \eqref{eqn:derivedequivalenceadjoint}-\eqref{eqn:fcrepres}, there is a projective crepant resolution  
 \begin{equation}
 \label{eqn:toriccrepresolution}
 f\colon Y:=\mathcal{M}(T)\longrightarrow X=\Spec R, 
 \end{equation}
 and the dual bundle $T^\vee\cong \bigoplus_{i\in Q_0} T_i^\vee$ determines a derived equivalence 
  \begin{equation}
 \label{eqn:toricderivedequivalenceadjoint}
 \Psi(-):= T^\vee\otimes_{A} - \colon D^b(A)\longrightarrow D^b(Y)
 \end{equation}
 which satisfies $\Psi(P_i) = T^\vee\otimes_{A} Ae_i\cong T_i^{-1}$ for $i\in Q_0$.

 Now that we have a multigraded linear series, we can consider the effect of cornering. For any cornering subset $C \subset Q_0$ containing $\{0\}$, Proposition~\ref{prop:functor} gives a commutative diagram
\begin{equation}
\label{eqn:morphismoverSpecR}
\xymatrix{
 Y=\mathcal{M}(T)\ar[dr]\ar[rr]^{g_C} & & \mathcal{M}(T_C)\ar[dl]\\
  & X & 
 }
\end{equation}
 where the image of $g_C$ is determined by $\bigotimes_{i\in C} \det(T_i)$. We now describe the image of $g_C$ explicitly following Craw--Quintero V\'{e}lez~\cite{CrawQuinterovelez12}. Consider the collection
 \[
 \mathscr{E}:= \{M_i \mid i\in C\}
 \]
 of rank one reflexive sheaves on $X=\Spec R$. The \emph{quiver of sections} of $\mathscr{E}$ is the quiver $Q^\prime$ with vertex set $C$, and where the arrows from vertex $i$ to vertex $j$ correspond to those torus-invariant sections in $\Hom_{R}(M_i,M_j)$ that do not factor through $M_k$ for some $k\neq i, j$. Every arrow $a\in Q_1^\prime$ therefore determines a vector $\div_X(a)\in \ZZ^{d}$ in the lattice of torus-invariant divisors on $X$, and we extend this to a $\ZZ$-linear map $\div_X\colon \ZZ^{Q_1^\prime}\to \ZZ^d$ by sending the characteristic function for $a\in Q_1^\prime$ to $\div_X(a)$. Combine this with the incidence map $\inc$ of $Q^\prime$ to obtain a $\ZZ$-linear map 
 \[
 \pi_X:=(\inc,\div_X)\colon \ZZ^{Q_1^\prime}\longrightarrow \ZZ^{C}\oplus \ZZ^d.
 \]
 The semigroup algebra $\kk[\mathbb{N}^{Q_1^\prime}]$ is the coordinate ring of the space $\mathbb{A}_{\kk}^{Q_1^\prime}$ of representations of dimension vector $(1,\dots, 1)$, and the ideal $I_{\mathscr{E}} := (y^u-y^v\in \kk[\NN^{Q_1^\prime}]\mid u-v\in \Ker(\pi_X))$ cuts out an affine toric subvariety $\mathbb{V}(I_{\mathscr{E}})\subseteq \mathbb{A}^{Q_1^\prime}$ that is invariant under the action of the torus $G:= \Spec \kk[\ZZ^{C}]$ given by $(g_i)\cdot (w_a) = (g_{\head(a)}g_{\tail(a)}^{-1}w_a)$. For $\theta\in \Theta_{\vv_C}^+$, the GIT quotient 
 \[
 Y_\theta:= \mathbb{V}(I_\mathscr{E})\git_\theta T \subseteq \mathcal{M}(T_C)
 \]
 admits a projective birational morphism $\tau_\theta\colon Y_\theta\to X$ obtained by variation of GIT quotient \cite[Proposition~2.14]{CrawQuinterovelez12}. In fact we have the following result: 
 
\begin{proposition}
\label{prop:SekiyaTakahashi}
 Let $X$ be a Gorenstein affine toric threefold. For any $C\subseteq Q_0$ containing $\{0\}$, the image of $g_C$ is  the irreducible component $Y_\theta$ of $\mathcal{M}(T_C)$ that's birational to $X$.
\end{proposition}
\begin{proof}
 We prove first that $Y_\theta$ is an irreducible component of $\mathcal{M}(T_C)$. For each nontrivial path $p$ in $Q^\prime$, define the vector $v(p):=\sum_{a\in \supp(p)} \div(a)$ to be the sum of the vectors of each arrow in the path $p$. Then the ideal $J_{\mathscr{E}}:= (p^+-p^-\in \kk Q^\prime \mid p^{\pm}\text{ share the same tail, head and vector})$ satisfies $A_C\cong \kk Q^\prime/J_{\mathscr{E}}$ by \cite[Lemma~2.5]{CrawQuinterovelez12}. These paths in $Q^\prime$ determine a lattice
 \[
 \Lambda:= \big(v(p^+)-v(p^-)\in \ZZ^{Q_1^\prime} \mid p^+-p^-\in J_{\mathscr{E}}\big).
 \]
 The proof\footnote{In fact the proof in our context is slightly simpler: we have only one algebra $A_C$, so the sentence in line 13 of the proof of \cite[Lemma~3.14]{CrawQuinterovelez12} that transitions between two algebras is redundant.} of \cite[Lemma 3.14]{CrawQuinterovelez12} shows that $\Ker(\pi) = \Lambda$, and the proof of \cite[Theorem~3.15]{CrawQuinterovelez12} applies verbatim to give that $Y_\theta$ is the unique irreducible component of $\mathcal{M}(T_C)$ containing the $G$-orbit closures of points of $(\kk^\times)^{Q_1^\prime}$ in $\mathbb{A}_{\kk}^{Q_1^\prime}$.
 
  To see that $Y_\theta$ is the image of $g_C$, consider the collection $\mathscr{L}=\{T_i \mid i\in C\}$ of basepoint-free line bundles on $Y=\mathcal{M}(T)$. Following the construction of Craw--Smith~\cite{CrawSmith08}, the quiver of sections of $\mathscr{L}$ is $Q^\prime$, because we have algebra isomorphisms 
  \[
 \kk Q^\prime/J_{\mathscr{E}} \cong \End\left(\bigoplus_{i\in C} M_i\right) \cong A_C = e_C A e_C \cong e_C \End\left(\bigoplus_{i\in Q_0} T_i\right) e_C \cong \End\left(\bigoplus_{i\in C} T_i\right).
  \]
 In particular, each arrow $a$ in $Q^\prime$ determines a vector $\div_Y(a)\in \ZZ^{\Sigma(1)}$ in the lattice of torus-invariant divisors on the toric variety $Y$ with fan $\Sigma$, and we obtain a $\kk$-algebra homomorphism $\overline{\varphi}\colon \kk[\NN^{Q_1^\prime}]\to \kk[x_\rho \mid \rho\in \Sigma(1)]$ to the Cox ring of $Y$ by sending the variable $y_a$ to the monomial $x^{\div_Y(a)}$. Applying the $\Spec$ functor gives a morphism of affine varieties that is equivariant with respect to the actions of the algebraic tori $G$ and $\Spec \kk[\Cl(Y)]$; following \cite[Theorem~1.1]{CrawSmith08}, this morphism of affine schemes descends to give a morphism  
 \[
 \varphi_{\vert\mathscr{L}\vert} \colon Y\longrightarrow \mathcal{M}(T_C)=\mathcal{M}(A_C,\vv_C,\theta). 
 \]
 Since the tautological bundles $\mathscr{W}_i$ on $\mathcal{M}(T_C)$ satisfy $\varphi_{\vert \mathscr{L}\vert}^*(\mathscr{W}_i) = T_i$ for $i\in C$, it follows that $\varphi_{\vert \mathscr{L}\vert}$ coincides with the universal morphism $g_C$. To describe explicitly the image of $g_C$, combine the incidence map of $Q^\prime$ with the map $\div_Y\colon \ZZ^{Q_1^\prime}\to \ZZ^{\Sigma(1)}$ to obtain a $\ZZ$-linear map  
 \[
 \pi_Y:=(\inc,\div_Y)\colon \ZZ^{Q_1^\prime}\longrightarrow \ZZ^{C}\oplus \ZZ^{\Sigma(1)}
 \]
 and hence an ideal $I_{\mathscr{L}} := (y^u-y^v\in \kk[\NN^{Q_1^\prime}]\mid u-v\in \Ker(\pi_Y))$. Then \cite[Proposition~4.3]{CrawSmith08} shows that the image of $\varphi_{\vert \mathscr{L}\vert}$ is the GIT quotient $\mathbb{V}(I_{\mathscr{L}})\git_\theta G$. The $\ZZ$-linear map $\pi$ factors via $\pi_Y$, so $I_{\mathscr{L}}\subseteq I_{\mathscr{E}}$ and hence 
 \begin{equation}
 \label{eqn:GITquotients}
 Y_\theta= \mathbb{V}(I_{\mathscr{E}})\git_\theta G\subseteq \mathbb{V}(I_{\mathscr{L}})\git_\theta G = \im(\varphi_{\mathscr{L}}) = \im(g_C).
 \end{equation}
 Moreover, the GIT quotient $\mathbb{V}(I_{\mathscr{L}})\git_\theta G$ is irreducible and contains the $G$-orbit closures of points of $(\kk^\times)^{Q_1^\prime}$ in  $\mathbb{A}_{\kk}^{Q_1^\prime}$, so the inclusion from \eqref{eqn:GITquotients} is an equality. 
 \end{proof}

 The next example shows that $\mathcal{M}(T_C)$ need not be irreducible, or equivalently, the morphism $g_C$ need not be surjective in this context.

\begin{example}
\label{exa:1/3(1,1,1)}
For the cyclic subgroup $G\subset \SL(3,\CC)$ of type $\frac{1}{3}(1,1,1)$, an NCCR for the $G$-invariant ring $R:=\kk[x,y,z]^{G}$ can be presented as the McKay quiver with relations:
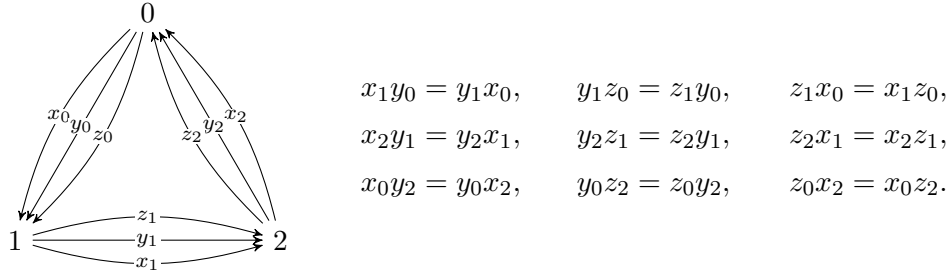
\begin{align*}
\begin{aligned}
\begin{tikzpicture}
\node (A) [draw=none,minimum size=4cm,regular polygon,regular polygon sides=3] at (0,0) {};
\node (M0) at (A.corner 1) {$0$};
\node (M1) at (A.corner 2) {$1$};
\node (M2) at (A.corner 3) {$2$};
\draw [->,bend left=-16] (M0) to node[gap] {$\scriptstyle{x_0}$} (M1);
\draw [->,bend left=-16] (M1) to node[gap] {$\scriptstyle{x_1}$} (M2);
\draw [->,bend left=-16] (M2) to node[gap] {$\scriptstyle{x_2}$} (M0);
\draw [->,bend left=0] (M0) to node[gap] {$\scriptstyle{y_0}$} (M1);
\draw [->,bend left=0] (M1) to node[gap] {$\scriptstyle{y_1}$} (M2);
\draw [->,bend left=0] (M2) to node[gap] {$\scriptstyle{y_2}$} (M0);
\draw [->,bend left=16] (M0) to node[gap] {$\scriptstyle{z_0}$} (M1);
\draw [->,bend left=16] (M1) to node[gap] {$\scriptstyle{z_1}$} (M2);
\draw [->,bend left=16] (M2) to node[gap] {$\scriptstyle{z_2}$} (M0);
\end{tikzpicture}
\end{aligned}
&\quad\quad
\begin{aligned}
&x_1y_0=y_1x_0, \quad &y_1z_0 =z_1y_0, \qquad &z_1x_0=x_1z_0, \\
&x_2y_1=y_2x_1, \quad &y_2z_1 =z_2y_1, \qquad &z_2x_1=x_2z_1, \\
&x_0y_2=y_0x_2, \quad &y_0z_2 =z_0y_2, \qquad &z_0x_2=x_0z_2. 
\end{aligned}
\end{align*}
The cornering subset $C=\{0,2\}$ corners away the vertex 1 to produce the algebra $A_{C}$, which can be presented as the quiver
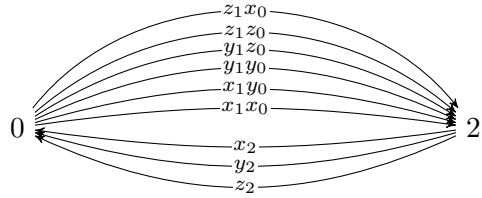
\begin{align*}
\begin{tikzpicture}
\node (M0) at (6,0) {$2$};
\node (M1) at (0,0) {$0$};
\draw [->,bend left=8] (M0) to node[gap] {$\scriptstyle{x_2}$} (M1);
\draw [->,bend left=16] (M0) to node[gap] {$\scriptstyle{y_2}$} (M1);
\draw [->,bend left=25] (M0) to node[gap] {$\scriptstyle{z_2}$} (M1);
\draw [->,bend left=8] (M1) to node[gap] {$\scriptstyle{x_1x_0}$} (M0);
\draw [->,bend left=16] (M1) to node[gap] {$\scriptstyle{x_1y_0}$} (M0);
\draw [->,bend left=25] (M1) to node[gap] {$\scriptstyle{y_1y_0}$} (M0);
\draw [->,bend left=33] (M1) to node[gap] {$\scriptstyle{y_1z_0}$} (M0);
\draw [->,bend left=41] (M1) to node[gap] {$\scriptstyle{z_1z_0}$} (M0);
\draw [->,bend left=52] (M1) to node[gap] {$\scriptstyle{z_1x_0}$} (M0);
\end{tikzpicture}
\end{align*}
with relations inherited from $A$. Consider the $A_C$ representation $(N,\lambda)$ of dimension vector $(1,1)$ defined by 
\[
\begin{bmatrix}
\lambda_{x_2} & \lambda_{y_2} & \lambda_{z_2}
\end{bmatrix}
=
\begin{bmatrix}
0 & 0 & 0 
\end{bmatrix}
\text{ and }
\begin{bmatrix}
\lambda_{x_1x_0} & \lambda_{x_1 y_0} & \lambda_{x_1 z_0} \\
\lambda_{y_1x_0} & \lambda_{y_1 y_0} & \lambda_{y_1 z_0} \\
\lambda_{z_1x_0} & \lambda_{z_1 y_0} & \lambda_{z_1 z_0} \\
\end{bmatrix}=
\begin{bmatrix}
1 & 0 & 0 \\
0 & 1 & 0 \\
0 & 0 & 1 
\end{bmatrix}.
\]
Then $j_!(N)$ is an $A$ representation $(V,\rho)$ with dimension vector $(1,3,1)$ defined by 
\begin{align*}
\begin{bmatrix}
\rho_{x_2} & \rho_{y_2} & \rho_{z_2}
\end{bmatrix} = \begin{bmatrix}
0 & 0 & 0
\end{bmatrix}, \, 
\begin{bmatrix}
\rho_{x_0} & \rho_{y_0} & \rho_{z_0}
\end{bmatrix} = \begin{bmatrix}
1 & 0 & 0 \\
0 & 1 & 0 \\
0 & 0 & 1 
\end{bmatrix}, \,
\begin{bmatrix}
\rho_{x_1} \\ \rho_{y_1} \\ \rho_{z_1}
\end{bmatrix} = \begin{bmatrix}
1 & 0 & 0 \\
0 & 1 & 0 \\
0 & 0 & 1 
\end{bmatrix}.
\end{align*}
This representation does not have $S_1$ as a submodule, and hence the point in $\mathcal{M}(E_{C})$ corresponding to $N$ is not in the image of $g_{C}$ by Proposition \ref{prop:submoduleSurjection}.
\end{example}
 
 In order to complete the statement and proof of Proposition~\ref{prop:introTakahashi}, we require a generalisation to our context of a notion introduced by Takahashi~\cite{Takahashi11}:
 
 \begin{definition}
A vertex $i\in Q_0$ is \emph{essential} if there exists a $0$-generated $A$-module that has a submodule isomorphic to $S_i$. Let $E\subset Q_0$ denote the set of essential vertices.
\end{definition}
 
 \begin{remark}
 \label{rem:essentialWalls}
 The proof of Bocklandt--Craw--Quintero V\'{e}lez~\cite[Proposition~4.7]{BCQV} shows that a vertex $i\in Q_0$ is essential if and only if the $0$-generated GIT chamber in $\Theta_{\vv}$ has a wall contained in the hyperplane $S_i^\perp:=\{\theta\in \Theta_{\vv} \mid \theta(S_i)=0\}$. This observation is the key ingredient in the next result which completes the proof of Proposition~\ref{prop:introTakahashi}.
 \end{remark}
 
  \begin{proposition}
 \label{prop:inessentials}
 Let $X$ be a Gorenstein affine toric threefold. For the set $C:= \{0\}\cup E$, the morphism $g_C\colon Y\to \mathcal{M}(T_C)$ is an isomorphism onto the irreducible component $Y_\theta$ in $\mathcal{M}(T_C)$.
 \end{proposition}
\begin{proof}
The stability parameter $\theta_C:= (\theta_i)\in \Theta_{\vv}$ given by 
\begin{equation}
\label{eqn:thetaC}
\theta_i = \left\{\begin{array}{cl} 1 & i\in C\backslash \{0\} \\ 0 & \text{otherwise}
\end{array}
\right.
\end{equation}
clearly lies in the closure of the $0$-generated chamber $\Theta_{\vv}^+$; in fact it lies in the interior of this chamber by Remark~\ref{rem:essentialWalls}. Therefore $\theta_C$ is $0$-generated, so if we choose this linearisation when constructing the GIT quotient $\mathcal{M}(T)$, then the polarising ample line bundle is 
 \[
 L(\theta_C):=\bigotimes_{i\in Q_0} T_i^{\otimes \theta_i} = \bigotimes_{i\in C} T_i.
 \]
 The result follows from Theorem~\ref{thm:multigradedlinearseries} and Proposition~\ref{prop:SekiyaTakahashi}.
\end{proof}

 \begin{remark}
 In the special case where $G\subset \SL(3,\CC)$ is a finite abelian subgroup, Proposition~\ref{prop:inessentials} recovers the main result of Takahashi~\cite{Takahashi11} by reconstructing the $G$-Hilbert scheme $Y$ as an irreducible component of the fine moduli space of $0$-generated $A_C=\End_Y(T_C)$-modules of dimension vector $\vv_C$. One cannot strengthen this conclusion, simply because the moduli space $\mathcal{M}(T_C)$ need not be irreducible in this case, as shown in Example~\ref{exa:1/3(1,1,1)}.  
 \end{remark}

 \section{On Reid's recipe and surface essentials}
 
 We continue to work under the assumptions of the previous section, where $X$ is a Gorenstein affine toric threefold and $Y\cong \mathcal{M}(T)$ is a crepant resolution. Our goal is to present a general framework and some explicit examples that are orthogonal in spirit to Proposition~\ref{prop:inessentials}, with a view to obtaining an isomorphism $\mathcal{M}(T)\cong \mathcal{M}(T_C)$. The terminology `essentials' might suggest that keeping such vertices is crucial if $g_C$ is to be an isomorphism, but this is not the case; indeed, here we choose which essential vertices to \emph{remove} from the set $Q_0$.
 
 Our motivation comes from comparing Corollary~\ref{cor:surjective}\two\ with the derived category statement of Reid's recipe~\cite{CautisLogvinenko09,Logvinenko10,CautisCrawLogvinenko12,BCQV}. The key observation is the following (see Remark~\ref{rem:essentialWalls} for another equivalent condition):
 
 \begin{lemma}
 \label{lem:essentialsRR}
 A nonzero vertex $i\in Q_0$ is essential iff the object $\Psi(S_i)$ in $D^b(A)$ is a sheaf.
 \end{lemma}
 \begin{proof}
 The proof of Bocklandt--Craw--Quintero V\'{e}lez~\cite[Lemma~4.2]{BCQV} shows that a nonzero vertex $i\in Q_0$ is essential if and only if there exists a torus-invariant $0$-generated $A$-module that has a submodule isomorphic to $S_i$. Now apply \cite[Theorem~1.1\one, Proposition~1.3]{BCQV}. 
 \end{proof}
 
 \begin{remark}
 One can also show (see \cite[Theorem~1.4]{BCQV}) that each nonzero inessential (= not essential) vertex $i\in Q_0$ is such that the class $[\Psi(S_i)]\in \Knumc(Y)$ is equal to $-[\mathcal{F}]$, where $\mathcal{F}$ is the class of a sheaf. In particular,  the statement of Corollary~\ref{cor:surjective}\two\ does not hold if we remove inessential vertices from $Q_0$, so it does not apply in the situation of Proposition~\ref{prop:inessentials}.
 \end{remark}
 
 We now restrict ourselves to the case when $X$ is isomorphic to the quotient singularity $\mathbb{A}^3_\kk/G$, where $G\subset \SL(3,\CC)$ is a finite abelian subgroup, and where $Y$ is the $G$-Hilbert scheme (see Remark~\ref{rems:Reidsrecipe}\one). Let $\Irr(G)$ denote the set of isomorphism classes of irreducible representation of $G$; this is the vertex set of the McKay quiver. Reid's recipe \cite{Reid97,Craw05} marks every proper, torus-invariant curve and surface in the $G$-Hilbert scheme $Y$ with an irreducible representation of $G$; those surfaces that are isomorphic to the del Pezzo surface of degree six are marked with two irreducible representations. 
 
 \begin{proposition}
 \label{prop:essentials}
 Let $C\subseteq \Irr(G)$ be a subset obtained by removing at most one  irreducible representation that marks each proper, torus-invariant surface in $\ghilb$. If $\dim_i j_!(N_x) \neq 0$ for all $i\in Q_0\setminus C$ and all $x \in Y'$, then the universal morphism $g_C\colon Y\to \mathcal{M}(T_C)$ is an isomorphism.
 \end{proposition}
 \begin{proof}
  Let $\theta_C\in \Theta_{\vv}$ denote the stability parameter defined as in \eqref{eqn:thetaC} above. Again we claim that the line bundle $L(\theta_C)\cong \bigotimes_{\rho\in C} T_\rho$ on $Y$ is ample. To see this, it suffices to show that for each torus-invariant curve $\ell\subset Y$, there exists $\rho\in C$ such that $\deg T_\rho\vert_\ell >0$. Reid's recipe labels the curve $\ell$ with some $\rho\in \Irr(G)$, and we have $\deg T_\rho\vert_\ell = 1$ by \cite[Lemma~7.2]{Craw05}. To see that $\rho\in C$, \cite[Corollary~4.6]{Craw05} states that every irreducible representation marks either a (collection of) curve(s) in $Y$ or a unique proper surface, and since we obtain $C$ from $Q_0=\Irr(G)$ by removing only representations that mark surfaces, we have $\rho \in C$ after all.
 
 Theorem~\ref{thm:multigradedlinearseries} and Proposition~\ref{prop:SekiyaTakahashi} imply that $g_C\colon Y\to \mathcal{M}(T_C)$ is an isomorphism onto the irreducible component $Y_\theta$ in $\mathcal{M}(T_C)$. Since we obtain $C$ by removing only essential vertices from $\Irr(G)$ , the object $\Psi(S_\rho)$ is a sheaf for each $\rho\in Q_0\setminus C$ by Lemma~\ref{lem:essentialsRR}. Since $\dim_i j_!(N_x) \neq 0$ for all $i\in Q_0\setminus C$ and all $x \in Y'$, Corollary~\ref{cor:surjective}\two\ shows that $g_C$ is an isomorphism.
 \end{proof}
 
  \begin{remarks}
  \label{rems:Reidsrecipe}
  \begin{enumerate}
  \item In Proposition~\ref{prop:essentials} we assumed that $X=\mathbb{A}^3_\kk/G$ (and hence $Y\cong \ghilb$) because \cite[Corollary~4.6]{Craw05} is known to hold only for the $G$-Hilbert scheme at present.
  \item The condition in Proposition~\ref{prop:essentials} that we remove from $\Irr(G)$ at most one irreducible representation that marks each surface implies that the indecomposable summands of $T_C$ generate the Picard group of $Y$; see \cite[Theorem~6.1]{Craw05}. Example~\ref{ex:DP6} illustrates that the statement of Proposition~\ref{prop:essentials} can fail without this condition.
  \end{enumerate}
  \end{remarks}

We now present several examples where we directly calculate the $A$ representation $j_!(N)=(V,\rho)$ in order to show when the hypotheses of Proposition~\ref{prop:essentials} hold. To calculate the vector spaces $V_k=e_kj_!(N)$ we present $e_kA e_C$ as a right $A_C$-module by a sequence 
  \[
A_C^{\oplus I} \xrightarrow{P} A_C^{\oplus J} \rightarrow e_k A e_C \rightarrow 0
\]
and apply $(-) \otimes_{A_C} N$ to produce a sequence
\[
N^{\oplus I} \xrightarrow{P_N} N^{\oplus J} \rightarrow e_k j_!(N) \rightarrow 0
\]
that presents the vector space $V_k$ as the cokernel of a matrix $P_N$.

\begin{example} \label{ex:1/6(1,2,3)}
 For the cyclic subgroup $G\subset \SL(3,\CC)$ of type $\frac{1}{6}(1,2,3)$, an NCCR for the $G$-invariant ring $R:=\kk[x,y,z]^{G}$ can be presented as the McKay quiver with relations:
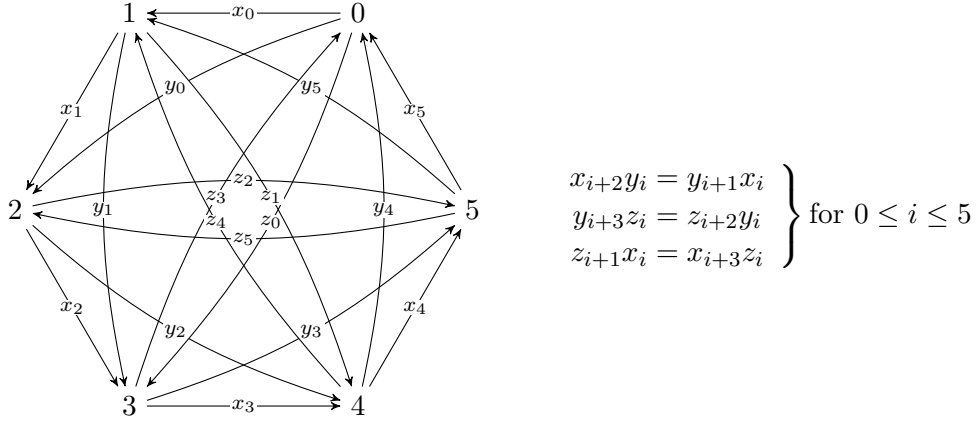
\begin{align*}
\begin{aligned}
\begin{tikzpicture}
\node (A) [draw=none,minimum size=6cm,regular polygon,regular polygon sides=6] at (0,0) {};
\node (M0) at (A.corner 1) {$0$};
\node (M1) at (A.corner 2) {$1$};
\node (M2) at (A.corner 3) {$2$};
\node (M3) at (A.corner 4) {$3$};
\node (M4) at (A.corner 5) {$4$};
\node (M5) at (A.corner 6) {$5$};
\draw [->,bend left=0] (M0) to node[gap] {$\scriptstyle{x_0}$} (M1);
\draw [->,bend left=0] (M1) to node[gap] {$\scriptstyle{x_1}$} (M2);
\draw [->,bend left=0] (M2) to node[gap] {$\scriptstyle{x_2}$} (M3);
\draw [->,bend left=0] (M3) to node[gap] {$\scriptstyle{x_3}$} (M4);
\draw [->,bend left=0] (M4) to node[gap] {$\scriptstyle{x_4}$} (M5);
\draw [->,bend left=0] (M5) to node[gap] {$\scriptstyle{x_5}$} (M0);
\draw [->,bend right=12] (M0) to node[gap] {$\scriptstyle{y_0}$} (M2);
\draw [->,bend right=12] (M1) to node[gap] {$\scriptstyle{y_1}$} (M3);
\draw [->,bend right=12] (M2) to node[gap] {$\scriptstyle{y_2}$} (M4);
\draw [->,bend right=12] (M3) to node[gap] {$\scriptstyle{y_3}$} (M5);
\draw [->,bend right=12] (M4) to node[gap] {$\scriptstyle{y_4}$} (M0);
\draw [->,bend right=12] (M5) to node[gap] {$\scriptstyle{y_5}$} (M1);
\draw [->,bend left=12] (M0) to node[gap] {$\scriptstyle{z_0}$} (M3);
\draw [->,bend left=12] (M1) to node[gap] {$\scriptstyle{z_1}$} (M4);
\draw [->,bend left=12] (M2) to node[gap] {$\scriptstyle{z_2}$} (M5);
\draw [->,bend left=12] (M3) to node[gap] {$\scriptstyle{z_3}$} (M0);
\draw [->,bend left=12] (M4) to node[gap] {$\scriptstyle{z_4}$} (M1);
\draw [->,bend left=12] (M5) to node[gap] {$\scriptstyle{z_5}$} (M2);
\end{tikzpicture}
\end{aligned}
\qquad
\begin{aligned}
\left.\begin{array}{c}
x_{i+2}y_i = y_{i+1}x_i \\
y_{i+3}z_i = z_{i+2}y_i \\
z_{i+1}x_i = x_{i+3}z_i \end{array}\right\} \text{for } 0\leq i\leq 5
\end{aligned}
\end{align*}
The cornering set $C=\{0,1,2,3,4\}$ removes the surface essential vertex 5, and the right $A_C$-module $e_5 A e_C$ is generated by $(x_4, y_3, z_2)$ and can be presented as the cokernel of the matrix
\[
P=\begin{bmatrix}
0 &  -z_1 &y_2 \\
z_0 & 0 & -x_2 \\
 -y_0  & x_1  & 0 
\end{bmatrix}.
\]
An $A_C$ representation $(N,\lambda)$ of dimension vector $(1,1,1,1,1)$ assigns a linear map $\lambda_p \in \textnormal{Mat}(\kk) \cong \kk$ to each path $p \in A_C$ such that  $\dim_5 j_!(N)= \dim \Coker P_N $ where 
\[
P_N:=
\begin{bmatrix}
0 & -\lambda_{z_1} &  \lambda_{y_2}  \\
\lambda_{z_0}  & 0 & -\lambda_{x_2}  \\
-\lambda_{y_0}& \lambda_{x_1}  & 0 
\end{bmatrix}.
\]
The matrix $P_N$ has determinant of  $\lambda_{x_1} \lambda_{y_2} \lambda_{z_0} - \lambda_{x_2} \lambda_{y_0} \lambda_{z_1}$, and it can be calculated using the relations on $A_C$ that
\begin{align*}
\begin{aligned}
\lambda_{x_0x_1} \det P &= \lambda_{x_0x_1}(\lambda_{x_1} \lambda_{y_2} \lambda_{z_0} - \lambda_{x_2} \lambda_{y_0} \lambda_{z_1}) \\
&= \lambda_{x_1} \lambda_{x_0x_1y_2} \lambda_{z_0} - \lambda_{x_0x_1x_2} \lambda_{y_0} \lambda_{z_1} \\
&= \lambda_{x_1} \lambda_{y_0x_2x_3} \lambda_{z_0} - \lambda_{x_1x_2} \lambda_{y_0} \lambda_{x_0z_1} \\
&= \lambda_{x_1x_2} \lambda_{y_0} \lambda_{z_0x_3} - \lambda_{x_1x_2} \lambda_{y_0} \lambda_{z_0x_3} \\
&= \lambda_{x_1} \lambda_{x_2} \lambda_{y_0} \lambda_{z_0x_3} - \lambda_{x_1x_2} \lambda_{y_0} \lambda_{z_0} \lambda_{x_3} \\
&=0
\end{aligned} 
\quad 
\begin{aligned}
\lambda_{y_0} \det P & = \lambda_{y_0}(\lambda_{x_1} \lambda_{y_2} \lambda_{z_0} - \lambda_{x_2} \lambda_{y_0} \lambda_{z_1}) \\
& = \lambda_{x_1y_2} \lambda_{y_0} \lambda_{z_0} - \lambda_{y_0x_2} \lambda_{y_0} \lambda_{z_1} \\
& = \lambda_{y_1x_3} \lambda_{y_0} \lambda_{z_0} - \lambda_{x_0y_1} \lambda_{y_0} \lambda_{z_1} \\
& = \lambda_{y_1} \lambda_{y_0} \lambda_{z_0x_3} - \lambda_{x_0y_1} \lambda_{y_0} \lambda_{z_1} \\
& = \lambda_{y_1} \lambda_{y_0} \lambda_{x_0}\lambda_{z_1} - \lambda_{x_0} \lambda_{y_1} \lambda_{y_0} \lambda_{z_1} \\
&=0.
\end{aligned}
\end{align*}
When $N$ is 0-generated there is a path $p$ from vertex 0 to vertex 2 such that $\lambda_p \neq 0$. All paths in $A$ from 0 to 2 either factor through vertex 5 or one of the paths $x_0x_1$ or $y_0$. As such $\det P =0$ and $\dim_5 j_!(N)>0$. It's easy to compute using Reid's recipe that $5$ is a surface essential vertex, so the $G$-Hilbert scheme is isomorphic to $\mathcal{M}(T_C)$ in this case by Proposition \ref{prop:essentials}.
\end{example}

\begin{remark}
 For the previous example, vertex 5 is the only essential vertex. It is natural to ask whether the morphism $g_C\colon Y\to \mathcal{M}(T_C)$ from Proposition~\ref{prop:SekiyaTakahashi} is always an isomorphism when we corner away essential vertices. The next example shows that this is not the case when we corner away a pair of essential vertices that mark the same proper torus-invariant surface according to Reid's recipe. The simplest example of the $G$-Hilbert scheme with this phenomenon is for an action of the group $(\ZZ/3\ZZ)^{\oplus 2}$, because in that case the $\ghilb$ contains a del Pezzo surface of degree $6$ (see \cite[Lemma~3.4]{Craw05}). Here we choose a simpler example defined by a consistent dimer model, so in this case we use Reid's recipe as described by Tapia Amador~\cite{TapiaAmador15}.
 \end{remark}

\begin{example}\label{ex:DP6}
 Let $\text{dP}_6$ denote the del Pezzo surface of degree six, let $Y$ denote the total space of the canonical bundle on $\text{dP}_6$, and write $\pi\colon Y\to \text{dP}_6$ for the contraction of the zero-section. There is a full strong exceptional collection of globally generated line bundles $(L_0,\dots, L_5)$ on $\text{dP}_6$, and the bundle $\bigoplus_{0\leq i\leq 5} \pi^*(L_i)$ is a tilting bundle on $Y$ whose endomorphism algebra can be presented by the following quiver with relations
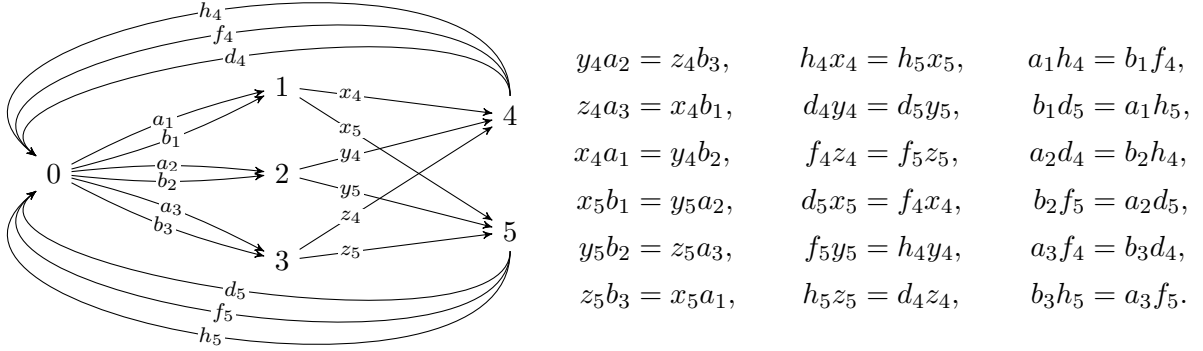
\begin{align*}
\begin{aligned}
\begin{tikzpicture}[yscale=0.77]
\path[use as bounding box] (-1,-2.8) rectangle (6,2.8);
\node (M0) at (0,0) {$0$};
\node (M1) at (3,1.5) {$1$};
\node (M2) at (3,0) {$2$};
\node (M3) at (3,-1.5) {$3$};
\node (M4) at (6,1) {$4$};
\node (M5) at (6,-1) {$5$};
\draw [->,bend left=8] (M0) to node[gap] {$\scriptstyle{a_1}$} (M1);
\draw [->,bend left=-8] (M0) to node[gap] {$\scriptstyle{b_1}$} (M1);
\draw [->,bend left=8] (M0) to node[gap] {$\scriptstyle{a_2}$} (M2);
\draw [->,bend left=-8] (M0) to node[gap] {$\scriptstyle{b_2}$} (M2);
\draw [->,bend left=8] (M0) to node[gap] {$\scriptstyle{a_3}$} (M3);
\draw [->,bend left=-8] (M0) to node[gap] {$\scriptstyle{b_3}$} (M3);
\draw [->,bend left=0] (M1) to node[gap, near start] {$\scriptstyle{x_4}$} (M4);
\draw [->,bend left=0] (M2) to node[gap,near start] {$\scriptstyle{y_4}$} (M4);
\draw [->,bend left=0] (M3) to node[gap, near start] {$\scriptstyle{z_4}$} (M4);
\draw [->,bend left=0] (M1) to node[gap, near start] {$\scriptstyle{x_5}$} (M5);
\draw [->,bend left=0] (M2) to node[gap, near start] {$\scriptstyle{y_5}$} (M5);
\draw [->,bend left=0] (M3) to node[gap, near start] {$\scriptstyle{z_5}$} (M5);
\draw [->,looseness=0.75, out=90, in=135] (M4) to node[gap] {$\scriptstyle{d_4}$} (M0);
\draw [->,looseness=1, out=90, in=135] (M4) to node[gap] {$\scriptstyle{f_4}$} (M0);
\draw [->,looseness=1.25, out=90, in=135] (M4) to node[gap] {$\scriptstyle{h_4}$} (M0);
\draw [->,looseness=0.75, out=-90, in=-135] (M5) to node[gap] {$\scriptstyle{d_5}$} (M0);
\draw [->,looseness=1, out=-90, in=-135] (M5) to node[gap] {$\scriptstyle{f_5}$} (M0);
\draw [->,looseness=1.25, out=-90, in=-135] (M5) to node[gap] {$\scriptstyle{h_5}$} (M0);
\end{tikzpicture}
\end{aligned}
& \qquad
\begin{aligned}
y_4a_2&=z_4b_3, \\
z_4a_3&=x_4b_1, \\
x_4a_1&=  y_4b_2, \\
x_5b_1&=y_5a_2, \\
y_5b_2&= z_5a_3, \\
z_5b_3&= x_5a_1,
\end{aligned}
 \qquad
\begin{aligned}
h_4x_4&=h_5x_5, \\
d_4y_4&=d_5y_5, \\
f_4z_4&=f_5z_5, \\
d_5x_5&=f_4x_4, \\
f_5y_5&=h_4y_4, \\
h_5z_5&=d_4z_4, 
\end{aligned}
\qquad
\begin{aligned}
a_1h_4&=b_1f_4, \\
b_1d_5&=a_1h_5, \\
a_2d_4&=b_2h_4, \\
b_2f_5&=a_2d_5, \\
a_3f_4&=b_3d_4, \\
b_3h_5&=a_3f_5.
\end{aligned}
\end{align*}
The cornering set $C:=\{ 0,1,2,3,4 \}$ removes the surface essential vertex $5$. An $A_C$ representation $(N,\lambda)$ of dimension vector $(1,1,1,1,1)$ assigns a linear maps $\lambda_p \in \kk$ to paths in $A_C$.  The right $A_C$-module $e_5Ae_C$ is generated by $(x_5,y_5,z_5)$ and the vector space $e_5 j_!(N)$ is given by the cokernel of the matrix
\[
P_N= 
\begin{bmatrix}
0 & -\lambda_{a_1} & \lambda_{b_1}\\
 \lambda_{b_2}  & 0 &  -\lambda_{a_2} \\
 -\lambda_{a_3}   & \lambda_{b_3} & 0
\end{bmatrix}
\]
 which has determinant $\lambda_{b_1} \lambda_{ b_2} \lambda_{b_3}-\lambda_{a_1} \lambda_{ a_2} \lambda_{ a_3}$. Using the relations on $A_C$ we calculate
\begin{align*}
 \lambda_{x_4} \det P  & = \lambda_{x_4}(\lambda_{b_1} \lambda_{ b_2} \lambda_{b_3}-\lambda_{a_1} \lambda_{ a_2} \lambda_{ a_3}) \\
 & = (\lambda_{b_1x_4} \lambda_{ b_2} \lambda_{b_3}-\lambda_{a_1x_4} \lambda_{ a_2} \lambda_{ a_3}) \\
  & = (\lambda_{a_3z_4} \lambda_{ b_2} \lambda_{b_3}-\lambda_{b_2y_4} \lambda_{ a_2} \lambda_{ a_3}) \\
    & = (\lambda_{a_3} \lambda_{ b_2} \lambda_{b_3z_4}-\lambda_{b_2} \lambda_{ a_2y_4} \lambda_{ a_3}) \\
        & = (\lambda_{a_3} \lambda_{ b_2} \lambda_{a_2y_4}-\lambda_{b_2} \lambda_{ a_2y_4} \lambda_{ a_3}) =0,
\end{align*}
and similarly $ \lambda_{y_4} \det P =0$ and $\lambda_{z_4} \det P =0$. When $N$ is 0-generated there is a path $p$ from vertex 0 to vertex 4 such that $\lambda_p \neq 0$. Any such path contains one of $x_4,y_4$ or $z_4$ so one of $\lambda_{x_4}, \lambda_{y_4}$, or $\lambda_{z_4}$ must be non-zero, so $\det(P)=0$ and hence $\dim \Coker P >0$. In particular $\dim_5 j_!(N) >0$ and by Proposition \ref{prop:essentials} the cornering morphism $g_C$ is surjective because 5 is a surface essential.

The cornering set $C'=\{0,1,2,3\}$ removes both vertices 4 and 5, and for an $A_{C'}$ representation $(N',\lambda')$ of dimension $(1)_{i \in C'}$ the vector space $e_5 j_!(N')$ is the cokernel of the matrix
\[ P'_{N'}= 
\begin{bmatrix}
0 & -\lambda'_{a_1} & \lambda'_{b_1}\\
 \lambda'_{b_2}  & 0 &  -\lambda'_{a_2} \\
 -\lambda'_{a_3}   & \lambda'_{b_3} & 0
\end{bmatrix}
\] However, as the vertex $4$ does not appear in $C'$ the stability condition on $N'$ does not require the existence of a non-zero path from $0$ to $4$ and the determinant of $P'_{N'}$ need not be zero. Indeed, the $A_{C'}$ representation $N'$ defined by $\lambda'_{a_1}=\lambda'_{a_2}=\lambda'_{a_3}=1$ and all other $\lambda'_p=0$ is 0-generated, the cokernel of $P'_{N'}$ is trivial, and $\dim_5 j_!(N')=0$. As such $g_{C'}$ is not surjective. 

 Note in this case that the tautological bundles $T_i$ on $Y$ satisfy $T_4\otimes T_5\cong T_1\otimes T_2\otimes T_3$, and removing both $T_4$ and $T_5$ leaves too few bundles to generate the Picard group of $Y$.  
\end{example}

\appendix

\section{The numerical Grothendieck group for compact support}
\label{sec:appendix}
 Here we recall the numerical Grothendieck group for compact support introduced in \cite{BayerCrawZhang16}. As an application, we extend to NCCRs satisfying Assumption~\ref{ass:KrullSchmidt} the connectedness result for the $G$-Hilbert scheme due to Bridgeland--King--Reid~\cite{BKR01}; this result was well-known to Van den Bergh~\cite[Remark~6.6]{VdB04}, but we give a slightly simpler proof using the numerical Grothendieck group. 

 Let $A$ be an associative $\kk$-algebra of the form $\kk Q/I$, where $Q$ is a finite, connected quiver and $I\subset \kk Q$ is a two-sided ideal generated by linear combinations of paths of length at least one. Each $i\in Q_0$ determines an indecomposable projective $A$-module $P_i:= Ae_i$. Our assumption on $I$ ensures that there is also a one-dimensional simple $A$-module $S_i=\kk e_i$ on which the class in $A$ of each arrow acts as zero. Let $\Kfin(A)$ denote the Grothendieck group of the bounded derived category $\Dfinb(A)$ of finite-dimensional $A$-modules. Since $A$ has finite global dimension, the Euler form $\chi_A\colon K(A)\times \Kfin(A)\to \ZZ$ is the bilinear form $\chi_A(E,F) = \sum_{i\in \ZZ} (-1)^i \dim_\kk \Ext^i_A(E,F)$. The \emph{numerical Grothendieck group} for $A$ is defined to be the quotient 
 \[
 \Knum(A):= \Kfin(A)/K(A)^\perp.
 \]
 
 \begin{lemma}
 \label{lem:KnumANCCR}
 If the classes $[P_i]$ for $i\in Q_0$ generate $K(A)$, then $K(A)\cong \bigoplus_{i\in Q_0} \ZZ[P_i]$, $\Knum(A)\cong \bigoplus_{i\in Q_0} \ZZ[S_i]$, and $\chi_A$ descends to a perfect pairing $\chi_A\colon K(A)\times \Knum(A)\to \ZZ$.
 \end{lemma}
 \begin{proof}
 The assumption ensures that the classes $[P_i]$ for $i\in Q_0$ generate $\Knum(A)^\vee$. The result is immediate from \cite[Lemma~7.1.1]{BayerCrawZhang16}. 
 \end{proof}
 
  To describe the numerical Grothendieck group on the geometric side, let $Y$ be a smooth scheme that is projective over an affine scheme of finite type, and let $T$ be a tilting bundle on $Y$ with endomorphism algebra $A:=\End_Y(T)$. Buchweitz--Hille~\cite[Proposition~2.6]{BuchweitzHille13} show that $T^\vee$ is also a tilting bundle with $\End_Y(T^\vee)\cong A^{\text{op}}$, so 
  \begin{equation}
  \label{eqn:NCCRequivPhi}
 \Phi(-)=\mathbf{R}\!\Hom_Y(T^\vee,-)\colon D^b(Y) \longrightarrow D^b(A)
 \end{equation}
 is an equivalence of derived categories with quasi-inverse 
  \begin{equation}
  \label{eqn:NCCRequivPsi}
 \Psi(-) = T^\vee\otimes_{A} -\colon D^b(A)\longrightarrow D^b(Y).
 \end{equation}
 Any such algebra $A$ can be presented as a quotient $\kk Q/I$, but the generators of $I$ may involve idempotents as noted in Remark~\ref{rem:idempotents}; to avoid this, we assume that the generators of $I$ are linear combinations of paths of length at least one so that Lemma~\ref{lem:KnumANCCR} holds.
 
 Let $K(Y)$ and $K_c(Y)$ denote the Grothendieck groups of the categories $D^b(Y)$ and $D^b_c(Y)$ respectively. Since $Y$ is smooth, the Euler form $\chi_Y\colon K(Y)\times K_c(Y)\to \ZZ$ is the bilinear form given by $\chi_Y(E,F) = \sum_{i\in \ZZ} (-1)^i \dim_\kk \Ext^i_{\cO_Y}(E,F)$, and the \emph{numerical Grothendieck group for compact support} on $Y$ is defined to be the quotient
 \[
 \Knumc(Y):= K_c(Y)/K(Y)^\perp.
 \]
 The geometric analogue of Lemma~\ref{lem:KnumANCCR} follows by applying the equivalence $\Psi$ from \eqref{eqn:NCCRequivPsi}:
 
 \begin{lemma}
 \label{lem:KnumcY}
We have that $K(Y)\cong \bigoplus_{0\leq i\leq n} \ZZ[T_i^\vee]$ and $\Knumc(Y)\cong \bigoplus_{0\leq i\leq n} \ZZ[\Psi(S_i)]$, and the Euler form $\chi_Y$ descends to a perfect pairing $\chi_Y\colon K(Y)\times \Knumc(Y)\to \ZZ$.
 \end{lemma}
 \begin{proof}
 The equivalence $\Psi$ induces isomorphisms $K(A)\cong K(Y)$ and $\Knum(A)\cong \Knumc(Y)$ by \cite[Theorem~7.2.1]{BayerCrawZhang16}. We compute $\Psi(P_i) = T^\vee\otimes_AAe_i = T_i^\vee$ for $i\in Q_0$, and the result is immediate from Lemma~\ref{lem:KnumANCCR}. 
  \end{proof} 

We now generalise the statement and proof of Bridgeland--King--Reid~\cite[Section~8]{BKR01}.
 
\begin{proposition}
 \label{prop:connected}
 Let $A$ be an \emph{NCCR} satisfying Assumption~\ref{ass:KrullSchmidt}. For any generic $\theta\in \Theta_{\vv}$, the moduli space $\cM(A,\vv,\theta)$ is connected. 
 \end{proposition}
 \begin{proof}
Suppose there exists a $\theta$-stable $A$-module $M$ of dimension vector $\vv$ that is not of the form $\Phi(\cO_y)$ for some closed point $y\in Y$. Write $E\in D^b_c(Y)$ for the object satisfying $\Phi(E)=M$. We claim that $E$ is quasi-isomorphic to a complex of locally-free sheaves of the form 
 \[
 L^{-2}\stackrel{g}{\longrightarrow} L^{-1},
 \]
 where $L^i$ is in degree $i$.  For this, let $M_1, M_2$ be non-isomorphic $\theta$-stable $A$-modules of dimension vector $\vv$. Any nonzero $\phi\in \Hom_A(M_1,M_2)$ has nonzero kernel and a proper image for dimension vector reasons. Then $\theta(\ker(\phi))>0$ because $M_1$ is $\theta$-stable, but then  $\theta(\im(\phi))<0$ because $M_2$ is $\theta$-stable. This is absurd, so $\Hom_A(M_1,M_2)=0$. Since $A$ is CY3, the category $D^b(A)$ has trivial Serre functor and hence $\Ext_A^3(M_1, M_2)=\Hom_A(M_2, M_1)=0$. Thus, for any $y\in Y$
 \[
 \Hom^i_{D^b(Y)}(E,\cO_y) = \Ext^i_A\big(M,\Phi_\theta(\cO_y)\big) = 0
 \]
 unless $1\leq i\leq 2$. The claim now follows from Bridgeland--Maciocia~\cite[Proposition~5.4]{BridgelandMaciocia02}.  Since $E\in D^b_c(Y)$, we have that $H^{-1}(E)\cong \ker(g)$ is a torsion subsheaf of $L^{-2}$ and hence is zero. It follows that $E\cong \cok(g)[1]$, and hence $[E]=-[\cok(g)]\in \Knumc(Y)$. 
 
 For any closed point $y\in Y$, the $\theta$-stable $A$-modules $\Phi(\cO_y)$ and $M$ have dimension vector $\vv$. This gives  $[\Phi(\cO_y)] = \vv = [M]\in \Knum(A)$, so  
 \[
 [\cO_y] = \vv = [E]\in \Knumc(Y),
 \]
 and hence $[\cO_y] = -[\cok(g)]\in \Knumc(Y)$. Let $j\colon Y\to\overline{Y}$ be a smooth projective completion. By \cite[Proof of Lemma 5.1.1]{BayerCrawZhang16}, we may pushforward numerical classes with compact support to obtain $[j_*\cO_y] = -[j_*\cok(g)] \in \Knum(\overline{Y})$. However, for any sufficiently ample bundle $L$ on $\overline{Y}$, the integers $\chi(j_*\cO_y\otimes L)$ and  $\chi(j_*\cok(g)\otimes L)$ are both positive, a contradiction. 
 \end{proof}
 
 \begin{corollary}
 For generic $\theta\in \Theta_{\vv}$, the fine moduli space $\cM(A,\vv,\theta)$ is a projective crepant resolution of $\Spec R$, and the bundle $T^\vee$ dual to the tautological bundle on $\cM(A,\vv,\theta)$ determines the derived equivalences \eqref{eqn:NCCRequivPhi} and \eqref{eqn:NCCRequivPsi}. 
 \end{corollary}
 
For $A=\End_R(M)$ satisfying Assumption~\ref{ass:KrullSchmidt}, Lemma~\ref{lem:KnumcY} establishes that the number of non-isomorphic indecomposable summands of the reflexive module $M$ is equal to the rank of the Grothendieck group of any projective crepant resolution of $\Spec R $. This statement is false without Assumption~\ref{ass:KrullSchmidt}; a counterexample can be constructed using \cite[Lectures on Noncommutative Resolutions, Example 2.17]{WemyssLecturesNCResolutions}.

\bibliographystyle{alpha}

\newcommand{\etalchar}[1]{$^{#1}$}

\end{document}